\documentclass[12pt]{amsart}

\usepackage{amsmath,amssymb}
\usepackage{amsfonts}
\usepackage{amsthm}
\usepackage{latexsym}
\usepackage{graphicx}
\usepackage{epsfig}
\usepackage{enumitem}
\usepackage{xcolor}

\newtheorem{theorem}{Theorem}[section]
\newtheorem{lemma}[theorem]{Lemma}
\newtheorem{proposition}[theorem]{Proposition}
\newtheorem{definition}[theorem]{Definition}
\newtheorem{corollary}[theorem]{Corollary}

\newtheorem{prop}{Proposition}[section]
\newtheorem{remark}[prop]{Remark}

\makeatletter \@addtoreset{equation}{section} \makeatother

\def\ddt{\frac{d}{dt}}
\def\ppt{\frac{\partial}{\partial t}}

\def\RR{{\mathrm R}}
\def \2R{{\hat{\RR}}}
\def\WW{{\mathrm W}}

\def\Rc{{\mathrm {Rc}}}
\def\oRc{{\overline{\mathrm {Rc}}}}

\def\SS{{\mathrm S}}

\def\He{\mathrm {Hess}}

\def\id{\mathrm{Id}}

\def\lie{\mathcal{L}}
\def\xi{\partial_{x_i}}
\def\yi{\partial_{y_i}}
\def\tr{\text{tr}}

\usepackage[margin=1.19 in]{geometry}

\begin{document}

\title{K\"{a}hler Gradient Ricci Solitons with Large Symmetry}



\author{Hung Tran}



\date{\today}



\begin{abstract} Let $(M, g, J, f)$ be an irreducible non-trivial K\"{a}hler gradient Ricci soliton of real dimension $2n$. We show that its group of isometries is of dimension at most $n^2$ and the case of equality is characterized. As a consequence, our framework shows the uniqueness of $U(n)$-invariant K\"{a}hler gradient Ricci solitons constructed earlier. There are corollaries regarding the groups of automorphisms or affine transformations and a general version for almost Hermitian GRS. The approach is based on a connection to the geometry of an almost contact metric structure.   
	
\end{abstract}
\maketitle

\section{Introduction}
Let $(M^m, g)$ be an orientable connected Riemannian manifold. In \cite{H3}, R. Hamilton introduced the Ricci flow equation, for $\Rc$ denoting the Ricci curvature,
\begin{equation}
	\label{rf}
	\ppt g(t)= -2 \Rc(t). \end{equation}
The theory has been utilized to solve fundamental problems; see \cite{perelman1, perelman2, perelman3, bohmwilking, bs091, bs072}. 
As a weakly parabolic system, it generically develops singularities and the study of such models is essential in any potential applications. Gradient Ricci solitons (GRS) are self-similar solutions to (\ref{rf}) and arise naturally in that context. Consequently, there have been numerous efforts to study them; see \cite{Hsurface, chow, perelman2, naber07, munse09, caozhou10, pewy10, b12rot, brendle14rotahigh, caotran1, LNW16fourPIC, MW17compact} and references therein.

A GRS $(M, g, f)$ is a Riemannian manifold such that, for a constant $\lambda$,  
\begin{equation}
	\label{grs}
	\Rc+\text{Hess}{f}=\lambda g. 
\end{equation}
 It is called shrinking, steady, or expanding depending on the sign of $\lambda$ being positive, zero, or negative. Clearly, any Einstein manifold is an example with $\text{Hess}{f}\equiv 0$ and $\lambda$ being the Einstein constant. Moreover, the Gaussian soliton refers to $(\mathbb{R}^{m}, g_{Euc}, \lambda \frac{|x|^2}{2})$ for the Euclidean metric $g_{Euc}$. 
 Naturally a soliton is called rigid if it is isometric to a quotient of $N^{m-k}\times \mathbb{R}^k$ with $f= \frac{|x|^2}{2}$ on the Euclidean factor. Furthermore, \cite{PW09grsym} shows that if the metric of a  GRS is reducible, the soliton structure is reducible accordingly. Thus, a soliton is called non-trivial (or non-rigid) if at least a factor in its de Rham decomposition is non-Einstein.
 
 
 Many non-trivial examples are K\"ahler and the topic receives tremendous interest; see, for examples, \cite{tian1, WZ04toric, caohd09, CZ12Kahler, MW15topo, CS2018classification, CF16conical, Kotschwar18Kahlercone, CD2020expanding, DZ2020rigidity}. In particular, significant efforts lead to the classification of all K\"{a}hler Ricci shrinker surfaces \cite{CDSexpandshriking19, CCD22finite, BCCD22KahlerRicci, LW23}. In an even dimension, $(M^{2n}, g, J)$ is called an almost Hermitian manifold if $g$ is compatible with an almost complex structure $J: TM\mapsto TM$, $J^2=-\id$. The associated K\"{a}hler two-form is defined as, for tangential vector fields $X$ and $Y$, 
 \[\omega(X, Y)=g(X, JY).\]
 Consequently, the structure is called almost K\"{a}hler if $\omega$ is closed and K\"{a}hler if, additionally, $J$ is a complex structure. On a K\"{a}hler manifold, the popular convention is to write $()^n$ for complex dimension $n$. A K\"{a}hler GRS $(M, g, J, f)$ is simultaneously a K\"{a}hler manifold and a gradient Ricci soliton.

 In this paper, we propose an investigation based on the group of symmetry. An isometry on $(M^m, g)$ is a diffeomorphism preserving the metric $g$. The dimension of the group of isometries is at most $\frac{m(m+1)}{2}$ \cite{KNvolumeI96} and it is attained if and only if the manifold is isometric to one of the following models of constant curvature: the round spheres, the real projective space, the Euclidean space, or the hyperbolic space. 



For an almost Hermitian manifold of real dimension $2n$, using the terminology of \cite{KNauto57}, an automorphism is an isometry which preserves $J$. S. Tanno \cite{tanno69Hermitian} showed that the maximal dimension of the automorphism group is $n(n+2)$ and  the maximal case is characterized to be homothetic to one of the followings: the unitary space $\mathbb{C}^n=\mathbb{R}^{2n}$ with $g_{Euc}$, a complex projective space $\mathbb{CP}^n$ with a Fubini-Study metric, or an open ball $B^n_{\mathbb{C}}$ with a Bergman metric. These models play an important role in our work. 
\begin{definition}
	Let $\mathbb{N}(k)$ be a simply connected K\"{a}hler manifold with constant holomorphic sectional curvature and normalized Ricci curvature $k$. 
\end{definition}

 From the above discussion, it is immediate that a Gaussian soliton on $\mathbb{C}^n$ has $n(2n+1)$ isometries and $n(n+2)$ automorphisms. Also, P. Petersen and W. Wylie showed that a homogeneous GRS must be rigid \cite{PW09grsym}. It is, thus, interesting to ponder the next best scenario. 
One observes that, many non-trivial K\"{a}hler GRS's, see \cite{caohd96, caohd97limits, koi90, CV96, fik03}, are $U(n)$-invariant and $\text{dim}(U(n))=n^2$. According to \cite{DW11coho}, their metrics all belong to the following cohomogeneity one structure: 
 
\textit{ \textbf{An Ansatz:} Let ${N}^{n-1}(k)$ be a K\"{a}hler-Einstein manifold with complex dimension $n-1$ and $\Rc_N=k \id$, $I$ be an interval, and functions $H, F: I\mapsto \mathbb{R}^+$. $(P, g_t)$ is a Riemann submersion of a line or circle bundle with coordinate $z$ over $({N}, F^2 g_N)$ and a bundle projection $\pi: P\mapsto \mathbb{N}$. $\eta$ is the one-form dual of $\partial_z$ such that $d\eta=q\pi^\ast \omega_{\mathbb{N}}$ for $q\in \mathbb{Z}$. If ${N}=\mathbb{CP}^{n-1}$ and $q=1$ then $P=\mathbb{S}^{2n-1}$ and one recovers the Hopf fibration. If ${N}=\mathbb{N}\neq \mathbb{CP}^{n-1}$, then one assumes the bundle is trivial. The metric $g$ and almost complex structure $J$ on $I\times P$ is given by} 
     \begin{align}
     	 \label{ansatz}
     	g &= dt^2+g_t=dt^2+ H^2(t) \eta\otimes \eta + F^2(t)\pi^\ast g_\mathbb{N},\\
     	J &= \partial_t \otimes H\eta - \frac{1}{H}\partial_z\otimes dt +\pi^\ast J_N.\nonumber
     \end{align}

 Our first result asserts the uniqueness. 
 

\begin{theorem}
	\label{main0}
	Let $(M^{n}, g, J, f, \lambda)$ be an irreducible non-trivial complete K\"{a}hler GRS. Its group of isometries is of dimension at most $n^2$ and equality happens iff it is smoothly constructed by the ansatz \ref{ansatz} for  ${N}=\mathbb{N}(k)$, $q=1$, and $f$ can be chosen as a function of $t$. If $\lambda\geq 0$ then $\mathbb{N}=\mathbb{CP}^{n-1}$.
\end{theorem}


\begin{remark} The construction of complete K\"{a}hler GRS with $N=\mathbb{CP}^{n-1}$  was considered by many authors and a systematic analysis is given in \cite{DW11coho}. In particular, the isometry group is $U(n)$ and there must be exactly one or two singular orbits (two only if $\lambda >0$). To smoothly compactify each, one must collapse either the whole sphere (both $H$ and $F$ going to zero) or just the fiber ($H$ going to zero). Here are all possible configurations:   
	\begin{itemize}
		\item $\lambda=0$, $I=[0,\infty)$, the singular orbit is either a point ($M$ is topologically $\mathbb{C}^n$) or $\mathbb{CP}^{n-1}$ (($M$ is $\mathbb{C}^n$ blowing up at one point) \cite{caohd96, CV96}. 
		\item $\lambda<0$, $I=[0,\infty)$, the original construction is due to \cite{caohd97limits, CV96, fik03}.  
		\item $\lambda>0$, $I=[0, 1]$, each singular orbit is $\mathbb{CP}^{n-1}$ \cite{koi90, caohd96, CV96}. 
		\item $\lambda>0$, $I=[0, \infty)$, the singular orbit is either a point or $\mathbb{CP}^{n-1}$ \cite{fik03}.
	\end{itemize}
For $\lambda<0$, it is possible that $k\leq 0$; see \cite[Remark 4.22]{DW11coho} and \cite[Theorem 1]{PTV99quasi}. 
\end{remark}
\begin{remark}
	 The metric in Theorem \ref{main0} has each $(P, g_t)$ being a deformed homogenous Sasakian structure with constant holomorphic sectional curvature. 
\end{remark}


For the reducible case, the group of isometries is potentially skewed by a Gaussian soliton factor of a large dimension. Thus, it is natural to consider the following. 

\begin{corollary}
	\label{main3}
	Let $(M^{n}, g, J, f, \lambda)$ be a complete simply connected non-trivial K\"{a}hler GRS. Its group of automorphisms is of dimension at most $n^2$ and equality happens iff it is either irreducible as in Theorem \ref{main0} or isometric to
	\begin{enumerate}[label=(\roman*)]
	\item a product of an Euclidean soliton and a Hamilton's cigar ($\lambda=0$);     
		\item a product of $\mathbb{N}(k)$ ($k\leq 0$) with a complete K\"{a}hler expanding GRS in real dimension two ($\lambda<0$).
	\end{enumerate} 
\end{corollary} 
\begin{remark}
	There is a list of \textit{all} models of GRS in real dimension two \cite{BM15}. 
\end{remark}

Under certain conditions, an infinitesimal isometry is closely related to conformal \cite{schoen95conformal} and affine vector fields \cite{KNvolumeI96}. 
For example, following \cite[Chapter 9]{KNvolumeII96}, one recalls that a K\"{a}hler manifold is non-degenerate if the restricted linear holonomy group at $x\in M$ contains the endormorphism $J_x$ for an arbitrary $x\in M$. 

\begin{corollary}
	\label{nondegenrate}
	Let $(M^{n}, g, J, f)$ be a non-degenerate complete simply connected K\"{a}hler GRS. If $f$ is non-constant, then the group of affine transformations is of dimension at most $n^2$ and equality happens iff it is either irreducible as in Theorem \ref{main0} or a product of $\mathbb{N}(k)$ ($k<0$) with a complete K\"{a}hler expanding GRS in real dimension two. 
\end{corollary}


Indeed, Theorem \ref{main0} follows from a more general version for (possibly incomplete) almost Hermitian GRS $(M^{2n}, g, J, f, \lambda)$. These structures are compatible:
\[g(X, Y) = g(JX, JY)  \text{  and } \He f (X, Y) = \He f(JX, JY).\]
The group of symmetry is to preserve all $g$, $J$, and $f$.
\begin{theorem}
	\label{main1}
	Let $(M^{2n}, g, J, f)$ be an almost Hermitian GRS with symmetry group $G$. If $f$ is non-constant then $\text{dim}(G)\leq n^2$ and equality happens iff locally it is either 
	
	\begin{enumerate}[label=(\roman*)]
		\item constructed by the ansatz \ref{ansatz} for  ${N}=\mathbb{N}(k)$ and $q=0, 1$. 
		\item a product of a line/circle with a hyperbolic space. 
	\end{enumerate} 
\end{theorem}
\begin{remark} The metric in case (ii) above can be written as, for non-zero constants $H$ and $A$, 
\begin{align*}
	g = dt^2+g_t &=dt^2+ H^2 dz^2+ e^{2Az}g_{\mathbb{C}^{n-1}},\\
	\lambda=\frac{\partial^2 f}{\partial t^2}&=-2(\frac{A}{H})^2(n-1).
\end{align*} 
\end{remark}
\begin{remark}
	For Ansatz \ref{ansatz}, equation (\ref{grs}) is equivalent to an ODE system: 
	\begin{align}
		{\lambda} &=-\frac{H''}{H}-\frac{(2n-2)F''}{F}+f''=\frac{H^2 q^2(2n-2)}{F^4}-\frac{H''}{H}-\frac{2(n-1)H'F'}{HF}+f'\frac{H'}{H}\nonumber\\
		\label{ODEsasa}
		&=\frac{k}{F^2}-\frac{2H^2 q^2}{F^4}- \frac{F''}{F}-(2n-3)(\frac{F'}{F})^2-\frac{H'F'}{FH}+f'\frac{F'}{F}.
	\end{align}
	The almost K\"{a}hler condition is equivalent to 
	\[FF'= qH.\] 
	The metric could then be rewritten to a modified Calabi's ansatz; see Lemma \ref{almostKahler}. 
\end{remark}
\begin{remark}
	It is possible to construct local solutions for (\ref{ODEsasa}) giving (possibly incomplete) manifolds with maximal symmetry. For $\mathbb{N}=\mathbb{CP}^{n-1}$, generalized versions of (\ref{ODEsasa}) were investigated by \cite{DW11coho} and \cite{BDW15}.  
\end{remark}

\begin{remark}
	The Gaussian soliton $(\mathbb{R}^{2n}, g_{Euc}, f=\lambda \frac{|x|^2}{2},\lambda) \text{ for } \lambda\neq 0$ belongs to family $q=1$ with $P$ being the round sphere, $H=F=t$, and $k=2n$. For $\lambda=0$, the soliton $(\mathbb{R}^{2n}, g_{Euc}, f=ax_i+b)$ belongs to family $q=0$ with $P$ being the Euclidean space.  
\end{remark}
To illustrate the dimension $n^2$, let's consider the case of a Gaussian soliton on $\mathbb{C}^n$ for $\lambda\neq 0.$ The isometry group consists of $2n$ translations and $\frac{2n(2n-1)}{2}$ rotations. With a standard coordinate $\{x_i, y_i\}_{i=1}^n$, one specifies an almost complex structure such that
\[J(\partial_{x_i})= \partial_{y_i}, ~~~~ J(\partial_{y_i})= -\partial_{x_i}. \]
Then it is clear that not all rotations preserve this tensor field. That's why the automorphism group is only of dimension $n(n+2)$. Among those, the translations do not preserve the potential function $f=\lambda \frac{|x|^2}{2}$. Consequently, the group of symmetry preserving $g$, $J$ and $f$ is of dimension $n^2$.

The paper is organized as follows. Section \ref{prelim} recalls general and useful preliminaries while Section \ref{coho1ansatz} is devoted to calculation about ansatz \ref{ansatz}. Afterward, we'll discuss the relation between the symmetry of an almost Hermitian GRS and one of its level sets determined by $f$. The key idea is that a symmetry group on $(M, g, J, f)$ induces a symmetry group of regular level sets considered as almost contact metric structures. In Section \ref{rigid}, the rigidity of a maximal dimension is examined and proofs of all theorems are collected. Finally the appendix explains our convention and recalls submersion. 

\subsection{Acknowledgment}  H. T was partially supported by grants from the Simons Foundation [709791], the National Science Foundation [DMS-2104988], and the Vietnam Institute for Advanced Study in Mathematics. We benefit greatly from discussion with Profs. McKenzie Wang, Catherine Serle, and Ronan Conlon. We also thank Will Wylie for suggestions and \cite{wyliepersonal}.

\section{Preliminaries}
\label{prelim}
We'll recall fundamental concepts and useful results about an almost complex structure, a gradient Ricci soliton, group actions on a manifold, an almost contact structure, and certain model spaces. 
\subsection{Almost Complex Structure}
Let $M$ be a smooth manifold of dimension $2n$. 
\begin{definition}
	An almost complex structure is a smooth section $J$ of the bundle of endormorphisms $\text{End}(TM)$ such that
	\[J^2=-\id. \]

\end{definition}
	One can immediately extend $J$ to be an endormorphism on the complexified tangent bundle $TM\otimes_{\mathbb{R}}\mathbb{C}$ via $\mathbb{C}$-linearity. An almost complex structure is said to be integrable if $M$ admits an atlas of complex charts with holomorphic transition functions such that $J$ corresponds to the induced complex multiplication on $TM\otimes_{\mathbb{R}}\mathbb{C}$. A real differentiable manifold with an integrable almost complex structure is, by definition, a complex manifold. Thanks to the work of Newlander and Nirenberg \cite{NN57}, the integrability of $J$ is equivalent to the vanishing of the Nijenhuis tensor
	\[N_J(X, Y)=[JX, JY]-[X, Y]-J[X, JY]-J[JX, Y]. \]
	
\begin{definition}
	Let $(M^{2n}, g)$ be a Riemannian manifold with an almost complex structure $J$. $(M, g, J)$ is called an almost Hermitian manifold and $g$ a Hermitian metric if
	\[g(JX, JY)=g(X, Y).\] 
	The fundamental $2$-form or K\"{a}hler form is given by 
	\[\omega(X, Y)= g(X, JY).\]
	$(M, g, J)$ is called almost K\"{a}hler if $d\omega=0$. When $J$ is integrable, we upgrade an almost Hermitian to Hermitian and almost K\"{a}hler to K\"{a}hler. 
\end{definition}	    
For a Riemannian manifold to be K\"{a}hler, the following is well-known.
\begin{proposition}\cite[Proposition 3.1.9]{BGbookSasakian08}
	Let $(M, g, J)$ be an almost Hermitian (real) manifold. The followings are equivalent:
	\begin{enumerate}
		\item $\nabla J=0$,
		\item $\nabla \omega_g=0$,
		\item $(M, g, J)$ is K\"{a}hler.  
	\end{enumerate}
\end{proposition} 
On a K\"{a}hler manifold, one observes that
\begin{align*}
	J(\RR(X, Y)Z) &=\RR(X, Y) JZ,\\
	\RR(X, Y, JZ, JW)=g(\RR(X, Y) JZ, JW) &= g(\RR(X, Y) Z, W)=\RR(X, Y, Z, W).
\end{align*}
Naturally, it leads to the notion of the Ricci form. 
\begin{definition}
	The Ricci form $\rho$ is the image of $\omega_g$ via the curvature operator:
	\begin{align*}
		\rho(X, Y) &=g(\RR(\omega_g)(X), Y).
	\end{align*}
\end{definition}

A priori, it is not clear how the Ricci form $\rho$ is related to the Ricci curvature tensor.
\begin{proposition}\cite[Proposition 2.45]{besse} On a K\"{a}hler manifold $(M, g, J, \omega_g)$, we have
	\[	\Rc(X, Y)=\rho(X, JY). \]
\end{proposition}

\begin{corollary}
	On a K\"{a}hler manifold $(M, g, J)$, $\Rc$ is $J$-invariant. 
\end{corollary}

\subsection{Gradient Ricci Solitons} 
In this subsection, we recall how a GRS is compatible with a complex setup. 
\begin{definition}
	$(M, g, J, f)$ is an almost Hermitian GRS if $(M, g, f)$ is a GRS, $(M, g, J)$ is an almost Hermitian manifold, and $\lie_{\nabla f} g$ is $J$-invariant.
\end{definition}
\begin{remark}
Because of (\ref{grs}), $\lie_{\nabla f} g$ is $J$-invariant if and only if $\Rc$ is $J$-invariant. Thus, the assumption is automatic for K\"{a}hler manifolds.
\end{remark}
\begin{definition}
	$(M, g, J, f)$ is a K\"{a}hler GRS if $(M, g, f)$ is a GRS, $(M, g, J)$ is K\"{a}hler manifold.
\end{definition}
In a complex coordinate system, the $J$-invariant property is equivalent to $\nabla f$ being a holomorphic vector field. That is, 
\[\lie_{\nabla f} g(\partial_{z_i}, \partial_{z_j})=\lie_{\nabla f} g(\partial_{\overline{z_i}}, \partial_{\overline{z_j}})=0. \]

\subsection{Group Actions on a manifold}
In this subsection, we review the basic setup and properties of group actions on a manifold. The main references are \cite{KNvolumeI96, KNvolumeII96, kobayashi95}. Let $G$ be a topological group. An action of $G$ on a manifold $M$ is a homomorphism from $G$ to the group of homomorphisms on $M$
\[g\mapsto A_g \text{ such that } A_g: M\mapsto M,  x\mapsto g.x\] 

The action is \textit{continuous/smooth} if the map $G\times M\mapsto M$, given by $(g, x)\mapsto g\cdot x$ is continuous/smooth (for smoothness, it requires G to be a Lie group). The action is said to be \textit{proper} if the associated map $G\times M \mapsto M\times M$, given by $(g, x)\mapsto (x, g\cdot x)$ is proper (that is, the inverse of any compact set is compact). 

For each $x\in G$, the subgroup $G_x=\{g\in G, g\cdot x=x\}$ is called the \textit{isotropy} subgroup or the \textit{stabilizer}. 
The orbit through $x$ is an immersed sub-manifold and there is a natural identification
\[G\cdot x=\{y\in M, y=g\cdot x, g\in G\} \equiv G/G_x. \]
Orbits are also classified based on the relative size of associated isotropy groups. In particular, principal orbits correspond to the smallest possible groups and singular ones have isotropy groups of higher dimensions.

At the infinitesimal level, a smooth vector field $X$ on $M$ generates a (local) one-parameter family of maps between domains in $M$. If the vector field is complete, then it generates global differmorphisms. If the corresponding maps preserve certain geometric quantities and structures then the vector field is called a (local) infinitesimal transformation of the same property. A vector field preserves a tensor $T$ if and only if 
\[\lie_X T=0.\]
 It is also noted that the set of all vector fields can be seen as a Lie algebra $\mathfrak{X}(M)$ by the natural bracket 
\[[X, Y]= XY-YX.\] 
Since 
\[\lie_{[X, Y]}=\lie_X\circ \lie_Y-\lie_Y\circ \lie_X,\]
the set of all infinitesimal transformations preserving a tensor $T$ is always a Lie sub-algebra of $\mathfrak{X}(M)$. It is noted that the group of transformations preserving a tensor $T$ is not necessarily a Lie group. 

Nevertheless, on a Riemnnian manifold, the group of isometries (preserving the Riemannian metric) is a Lie group \cite[Chapter 6. Theorem 3.4]{KNvolumeI96}. The infinitesimal transformation corresponding to a subgroup of isometries is called a Killing vector field. That is,
\[\lie_X g=0.\]
The Lie algebra of such complete vector fields corresponds to the Lie algebra of the Lie group of all isometries on $M$. It is well-know that a Killing vector field is totally determined by its zero and first order values at a point $(X_p, (\nabla X)_p)$ \cite[Chapter VI]{KNvolumeI96}. \\

The K\"{a}hler and GRS structures impose rigidity on the Riemannian manifold as the followings are well-known \cite{fik03}.   
\begin{lemma}
	\label{killing}
	Let $(M, g, J, f)$ be a K\"{a}hler gradient Ricci soliton. Then, we have the followings:
	\begin{enumerate}
		\item $J(\nabla f)$ is a Killing vector field.
		\item $\lie_{\nabla f} J\equiv 0$. 
	\end{enumerate}
\end{lemma}
On an almost Hermitian GRS $(M, g, J, f)$, one may consider transformations and vector fields preserving each individual structure: the metric $g$, the almost complex structure $J$, and the potential function $f$. The group of such symmetry is clearly a closed subgroup of the group of isometry and, thus, is a Lie group. 



\subsection{Almost Contact Structure}
In this subsection, we recall important notions about an almost contact structure following the book by C. Boyer and K. Galicki \cite{BGbookSasakian08}. 

\begin{definition}
	A $(2n+1)$-dimensional manifold $M$ is an almost contact manifold if there exists a triple $(\zeta, \eta, \Phi)$ where $\zeta$ is a vector field, $\eta$ is a $1$-form, $\Phi$ is a tensor field of type $(1, 1)$, and they satisfy, 	everywhere on $M$,  
	\[\eta(\zeta)=1 \text{  and } \Phi^2=-\id+\zeta\otimes \eta.\]

\end{definition}
\begin{definition} An almost contact manifold $(M, \zeta, \eta, \Phi)$ with a Riemannian metric $g$ is called an almost contact metric structure if 
	\[g(\Phi(X), \Phi(Y))=g(X, Y)-\eta(X)\eta(Y). \]
\end{definition}
\begin{definition}
	The holomorphic or $\Phi$-sectional curvature of an almost contact manifold $(M, \zeta, \eta, \Phi)$ is given by, for $\eta(X)=0$ and $g(X, X)=1$,
	\[K_\Phi(X)=K(X, \Phi(X)).\]
\end{definition}

Closely related is the notion of a contact structure. 
\begin{definition}
	A $(2n+1)$-dimensional manifold $M$ is a contact manifold if there exists a $1$-form $\eta$, called a contact $1$-form, on $M$ such that
	\[\eta \wedge (d\eta)^n \neq 0\]
	everywhere on $M$. A contact structure is an equivalence class of such $1$-forms. 
\end{definition}

\begin{definition}
	An almost contact metric structure $(M, \zeta, \eta, \Phi, g)$ is called a contact metric structure if one further assumes
	\[g(X, \Phi(Y))=d\eta(X, Y).\]
\end{definition}
It is immediate to check that a contact metric structure is indeed a contact manifold by the above definition. As $\zeta$ and $\Phi$ are uniquely determined by $\eta$ and $g$, we also denote a contact metric manifold by $(M, \eta, g)$.  

\begin{definition} A contact metric structure $(M, g, \eta)$ is called Sasakian if the cone $C(M)=M\times \mathbb{R}^+$ with the cone metric $r^2 g+ dr^2$ is K\"{a}hler. 
\end{definition}




Next we recall certain transformations which will play crucial roles.

\begin{definition}
	\label{Sasahomothety}
	Let $(M, \zeta,\eta, \Phi, g)$ be an almost contact metric structure. For $a>0$, a transverse $a$-homothety deformation is given by 
	\[ \hat{\zeta} =\frac{1}{a}\zeta, ~~~ \hat{\eta} =a\eta, ~~~ \hat{\Phi}=\Phi,~~~ \hat{g}= ag+(a^2-a)\eta\otimes \eta. \]
\end{definition}
If $(M, \zeta,\eta, \Phi, g)$ is Sasakian, then so is its homothety transformation. 
\begin{definition}
	\label{Sasadeform}
	Let $(M, \zeta,\eta, \Phi, g)$ be an almost contact metric structure. For $a>0$, an $\pm a$-deformation is given by 
	\[ \zeta^\ast =\zeta, ~~~ \eta^\ast =\eta, ~~~ \Phi^\ast=\pm\Phi,~~~ g^\ast= ag+(1-a)\eta\otimes \eta. \]
\end{definition}
 A $\pm a$-deformation of a Sasakian manifold is not necessarily Sasakian.

\subsection{Model Spaces}

Using the submersion toolkit, we can describe several model spaces that will appear in our classification. First, the {unitary space} is the complex formulation of the Euclidean space $\mathbb{C}^n=\mathbb{R}^{2n}$ with standard coordinates $\{x_1, y_1,..., x_n, y_n\}$. The metric, the almost complex structure, and the fundamental $2$-form are as follows:
\begin{align*}
	g&=\sum_i (d{x^i})^2+ (dy^i)^2,\\
	J &= \sum_{i} (\partial_{y_i}\otimes dx^i-\partial_{x_i}\otimes dy^i),\\
	\omega_{\mathbb{C}^n} &= -2\sum_i dx^i\wedge dy^i.
\end{align*}

\textbf{The flat Sasakian space $(P, g_P)=\mathbb{R}^{2n+1}(-3)$}: the total space of a real line bundle over $\mathbb{C}^n$ with coordinates $\{x_1, y_1,..., x_n, y_n, z\}$. For $\eta = dz+2\sum_i y_i dx_i$, one considers: 
\begin{align*}	
	g_P &=\sum_i \big(((d{x_i})^2+ (dy_i)^2 \big)+\eta\otimes \eta,\\ 
	\Phi &=\sum_i \big(\partial_{y_i}\otimes dx^i-(\partial_{x_i}-2y_i\partial_z)\otimes dy^i. \big).
\end{align*}
It is readily verified, by Lemma \ref{submersioncurvature}, that $(P, g_P, \eta, \partial_z, \Phi)$ is Sasakian with constant $\Phi$-sectional curvature $-3$.

\textbf{The spherical Sasakian $(P, g_P)=\mathbb{S}^{2n+1}(a)$}: For simplicity, we utilize the ambient coordinates of $\mathbb{R}^{2n+2}$, $\{x_1, y_1,..., x_{n+1}, y_{n+1}\}$. All tensors described below are understood as their restriction to the unit sphere.
With the induced metric, the canonical Sasakian structure on $\mathbb{S}^{2n+1}$ is given by
\begin{align*}
	\zeta &= \sum_i (y_i\partial_{x_i}-x_i\partial_{y_i}),\\
	\eta &= \sum_i (y_i dx^i- x_i dy^i),\\
	\Phi &=\sum_{i, j}(x_ix_j-\delta_{ij})\partial_{x_i}\otimes dy_j-(y_iy_j-\delta_{ij})\partial_{y_i}\otimes dx_j+ x_jy_i\partial_{y_i}\otimes dy_j- x_iy_j\partial_{x_i}\otimes dx_j
\end{align*}
Let $\pi: \mathbb{S}^{2n+1}\mapsto N=\mathbb{CP}^n$ be the Hopf fibration and $g_N$ the Fubini-Study metric. The Sasakian metric can be realized as 
\[g = \pi^\ast g_N +\eta\otimes \eta, ~~ d\eta =\pi^\ast \omega_N.\]
Via a homothetic deformation (Definition \ref{Sasahomothety}), if $g_N$ is scaled to have Ricci curvature $k\id$, $k>0$, then the constant $\Phi$-sectional curvature $a$ of $g_P$ is, by Lemma \ref{submersioncurvature},
\[a=\frac{4k}{n+1}-3>-3.\] 
%

%


\textbf{The hyperbolic Sasakian $(P, g_P)=\mathbb{SB}^{2n+1}(a)$}: Let $g_0$ be the Bergman metric of constant sectional curvature $-1$ in the unit ball in $\mathbb{C}^n$. One then scales it to have Ricci curvatutre $k\id$, for $k<0$ and denote such construction by $N=B^n_\mathbb{C}(k)$ with metric $g_N$. Let $\omega_N$ be the corresponding K\"{a}hler form and, since $B^n_\mathbb{C}(k)$ is simply connected, there exists $1$-form $\alpha$ such that $d\alpha=\omega_N$. On the total space of the line bundle $P=B^n_\mathbb{C}(k)\times \mathbb{R}$ with natural projection $\pi$, one considers: 
\begin{align*}
	g_P &= \pi^\ast g_N+\eta\otimes \eta,\\
	\eta &= dz+ \pi^\ast\alpha.
\end{align*} 
By Lemma \ref{submersioncurvature}, the $\Phi$-sectional curvature of $(P, g_P)$ is
\[a=\frac{k}{2n-1}-3<-3.\]

\begin{theorem}\cite{tanno69sasa}
	\label{tannoSasa}
	Let $(M^{2n+1}, g, \eta, \Phi, \zeta)$ be a simply connected complete Sasakian manifold with constant $\Phi$-sectional curvature $H$ then it must be isometric to:
	 \begin{enumerate}[label=(\roman*)]
	 	\item ($H>-3$) the Sasakian sphere $\mathbb{S}^{2n+1}(H)$, 
	 	\item ($H=-3$) the flat Sasakian space $\mathbb{R}^{2n+1}(-3)$,  
	 	\item ($H<-3$) the Sasakian disk model $\mathbb{SB}^{2n+1}(H)$.   
	 \end{enumerate} 
\end{theorem}

As described earlier, Sasakian manifolds belong to the family of almost contact metric structures which also include the following. For $P=\mathbb{R}\times \mathbb{C}^n$ and a constant $A$, 
\[g_P=dz^2+e^{2A z}g_{\mathbb{C}^{n-1}}.\]
One realizes it as a hyperbolic metric $\mathbb{H}^{2n+1}(-A^2)$. 
\begin{lemma}
	\label{curvatrehyper}
	The sectional and Ricci curvature of the hyperbolic metric $g_P$, for orthonormal vectors $X, Y$ on $\mathbb{C}^n$ and $\partial_z$ along $\mathbb{R}$,
	\begin{align*}
		K(\partial_z, X) &=-A^2 =K(X, Y)\\
		\Rc(\partial_z, \partial_z) &=-2nA^2=\Rc(X_i, X_i). 
	\end{align*}
\end{lemma}

All  above models appear in the following result. Let $(P^{2n+1}, g, \eta, \zeta, \Phi)$ be an almost contact metric structure. The symmetry group is to preserve both $g$, $\eta$, $\zeta$ and $\Phi$.  
\begin{theorem}\cite{tanno69} 
	\label{tannoclassificationalmost}
The maximum dimension of the symmetry group is $(n+1)^2$. It is attained iff the sectional curvature for $2$-planes which contain $\zeta$ is a constant $C$ and the manifold is one of the following spaces:
	\begin{enumerate} [label=(\roman*)]
		\item $C>0$: an $\pm b$ deformation of a homogeneous Sasakian manifold with constant $\Phi$-sectional curvature $H$ or, precisely,
		\begin{itemize}
			\item $H>-3:$ the Sasakian sphere $\mathbb{S}^{2n+1}(H)$ or its quotient by a finite group generated by $\text{exp}(t\zeta)$ for $2\pi/t$ being an integer, 
			\item $H=-3$: the flat Sasakian space $\mathbb{R}^{2n+1}(-3)$ or its quotient by a cyclic group generated by $\text{exp}(t\zeta)$, 
			\item $H<-3$: the Sasakian disk model $\mathbb{SB}(H)$ or its quotient by a cyclic group generated by $\text{exp}(t\zeta)$. 
		\end{itemize}
		\item $C=0$: six global Riemannian product $X\times \mathbb{CP}^{n-1}(k)$, $X\times \mathbb{C}^{n-1}$, $X\times B^{n-1}_{\mathbb{C}}(k)$ where $X$ is a line or a circle;
		\item $C<0$ the hyperbolic space $\mathbb{H}^{2n+1}(C)$. 
	\end{enumerate}
\end{theorem}
For a Sasakian model with submersion $\pi: P\mapsto N$ with $N=\mathbb{N}(k)$, the metric can always be written as 
\[g=g_N+ \eta\otimes \eta, ~~ d\eta=\pi^\ast \omega_N.\]

\begin{lemma}
	\label{curvaturedeform}
	If $(M, \zeta', \eta', \Phi', g')$ is obtained via a transverse $a$-homothety and an $\pm b$-deformation then
	\[g'=ba g_N+a^2 \eta\otimes \eta,~~ \zeta'=\frac{1}{a}\zeta, ~~~, \eta'=a \eta, ~~~\Phi'=\pm \Phi.\] 
\end{lemma}
\begin{proof}
	Via a transerver $a$-homothety transformation: 
	\begin{align*}
		g^\ast &=a g+(a^2-a)\eta\otimes \eta= a g_N+a^2 \eta\times \eta=a g_N+\eta^\ast \otimes \eta^\ast;\\
		\eta^\ast &= a \eta;~~	\zeta^\ast = \frac{1}{a}\zeta;~~		\Phi^\ast = \Phi.	
	\end{align*}
	Via a $\pm b$-deformation:
	\begin{align*}
		g' &=b g^\ast+ (1-b)\eta^\ast\otimes \eta^\ast= ba g_N+b \eta^\ast\otimes \eta^\ast+ (1-b)\eta^\ast\otimes \eta^\ast\\
		&= ba g_T+\eta^\ast\otimes \eta^\ast= ba g_N+a^2 \eta\otimes \eta\\
		\eta' &=\eta^\ast=a \eta,~~\zeta'=\zeta^\ast=\frac{1}{a}\zeta,~~
		\Phi' = (\pm)\Phi.  
	\end{align*}
\end{proof}

\section{Cohomogeneity One Ansatz}
\label{coho1ansatz}
Here we assume the cohomogeneity one symmetry and collect calculation related to ansatz \ref{ansatz}. The setup follows \cite{DW11coho} closely. Let $G$ be a Lie group acting isometrically on a Riemannian manifold $(M, {g})$. Supposed that there is a dense subset $M_0\subset M$ such that, locally, there is a $G$-equivariant diffeomorphism:
\[\Psi: I\times P\mapsto M_0 \text{ given by } \Psi(t, hK)=h\cdot \gamma(t).\]
Here, $I$ is an interval; $\gamma(t)$ is a unit speed geodesic intersecting all orbits orthogonally; $P=G/K$ where $K$ is the istropy group along $\gamma(t)$. It follows that 
\[\Phi^\ast ({g})={g}= dt^2+ {g}_t\]
where ${g}_t$ is a one-parameter family of $G$-invariant metrics on $G/K$.  
For unit vector fields $\nu=\Phi_{\ast}(\partial_t)$, 
let $L$ denote the shape operator
\begin{align*}
L(X) &= {\nabla}_X \nu.
\end{align*}

We will consider $L_t=L_{\mid_{\Psi(t\times P)}}$ to be a one-parameter family of endormorphisms on $TP$ via identification $T(\Psi(t\times P))=TP$. Following \cite{DW11coho}, one observes
\[\partial_t g= 2g_t\circ L_t .\]
Thanks to Gauss, Codazzi, and Riccati equations, the Ricci curvature of $(M_0, {g})$ is totally determined by the geometry of the shape operator and how it evolves. That is, for tangential vectors $X$ and $Y$, 
\begin{align}
	\Rc(X, Y) &=\Rc_t(X, Y)-\tr(L)g_t(LX, Y)-g_t({L'}(X), Y), \nonumber\\
	\label{Rccomp}
	\Rc(X, N) &=-\nabla_X \tr(L_t)-g(\delta L, X),\\ 
	\Rc(N, N) &= -\text{tr}({L'})-\text{tr}(L^2).\nonumber
\end{align}
Here $\Rc_t$ denotes the Ricci curvature of $(P, g_t)$, $\delta L=\sum_{i}\nabla_{e_i}L(e_i)$ for an orthonormal basis and $\tr{T}=\tr_{g_t}T_t$. \\

We are particularly interested in the metric given by Ansatz \ref{ansatz}. We recall 
\begin{align*}
	g=dt^2+ g_t&= dt^2+ F(t)^2 \pi^\ast g_N+ H(t)^2\eta\otimes \eta,\\
	\eta &= (dz+ q\pi^\ast\alpha), ~~d\alpha =\omega_{\mathbb{N}}.
\end{align*}  
Thus,
\[2g_t L_t= g'_t= 2\frac{H'}{H} H^2\eta\otimes \eta+2\frac{F'}{F} F^2\pi^\ast g_N.\]
For $\id$ denoting the identity operator on the horizontal subspace of $TP$, which is $g_t$- perpendicular to $\partial_z$,  
\begin{align*} 
	L_t &= \frac{H'}{H} \partial_z \otimes \eta + \frac{F'}{F}\id.\,\\
	{L'}_t &=\Big(\frac{H''}{H}-\big(\frac{H'}{H}\big)^2\Big) \partial_z \otimes {\eta}+\Big(\frac{F''}{F}-\big(\frac{F'}{F}\big)^2\Big)\id. 
\end{align*}
Consequently,
\begin{align*}
	\tr{L_t} &=\frac{H'}{H}+(2n-2)\frac{F'}{F},~~~ \tr{L^2_t} =(\frac{H'}{H})^2+(2n-2)\frac{(F')^2}{F^2},\\
	\tr{L'_t} &=\frac{H''}{H}+(2n-2)\frac{F''}{F}-\frac{(H')^2}{H^2}-(2n-2)\frac{(F')^2}{F^2}.
\end{align*}
A natural almost complex structure on $I\times P$ is constructed from one on $(N, g_N)$:
\[J = \partial_t \otimes H\eta - \frac{1}{H}\partial_z\otimes dt +\pi^\ast J_N.\]
Thus, the K\"{a}hler form becomes:
\begin{align*}
	\omega &= 2dt\wedge H\eta+ F^2 \pi^\ast\omega_N,\\
	d\omega &= -2qH dt \wedge \pi^\ast \omega_N+ 2FF' dt\wedge\pi^\ast\omega_N.
\end{align*}
Consequently, the metric is almost K\"{a}hler if and only if $FF'=qH.$

\begin{lemma}
	\label{ricciansatz}
	Let $(I\times P, g)$ be given as in ansatz \ref{ansatz}. Let $\{E_i\}_{i=1}^{2m}$ be horizontal lifts of eigenvectors of $\Rc_N$ to $(P, g_t)$. Then $\Rc_g$ is diagonalized as
	\begin{align*}
		\Rc(N, N) &=-\frac{H''}{H}-(2n-2)\frac{F''}{F},\\
		\Rc(\partial_z,{\partial_z}) &=H^2 \Big(\frac{H^2 q^2}{F^4}(2n-2)-\frac{H''}{H}-(2n-2)\frac{F'}{F}\frac{H'}{H} \Big),\\
		\Rc(E_i, E_i) 
		&=F^2 \Big(\frac{k}{F^2}-\frac{H^2 q^2}{F^4}2-\frac{F''}{F}-\frac{F'}{F}\frac{H'}{H}-(2n-3)\big(\frac{F'}{F}\big)^2 \Big).	
	\end{align*}
	If $(I\times P, g, J)$ is almost K\"{a}hler and $q=1$ then $FF'=H$ and
	\begin{align*}
		\Rc(N, N) &=-\frac{F'''}{F'}-(2n+1)\frac{F''}{F}=\Rc(\frac{\partial_z}{H},\frac{\partial_z}{H}),\\
		\Rc(E_i, E_i) 
		&=F^2 \Big(\frac{k}{F^2}-2n\big(\frac{F'}{F}\big)^2-2\frac{F''}{F}\Big).	
	\end{align*}
\end{lemma} 
\begin{proof}
It follows from  (\ref{Rccomp}), the calculation above, and Lemma \ref{submersioncurvature}.	

\end{proof}

If $f$ is invariant by the cohomogeneity one action, then equation (\ref{grs}) is reduced to
\begin{align}
	0 &=-(\delta L)-\nabla \tr{L},\nonumber \\
	\label{reducedsolitionsystem}
	\lambda &=-\text{tr}(L')-\text{tr}(L^2)+f'',\\
	\lambda g_t(X, Y) &=\Rc_t(X, Y)-(\tr{L}) g_t(L X, Y)-g_t({L'}(X), Y) +f' g_t(LX, Y). \nonumber
\end{align}


\begin{lemma}
	\label{solitonsubmer1}
	Let $(I\times P, g)$ be given as in ansatz \ref{ansatz} and $f=f(t)$ then the GRS equation (\ref{grs}) becomes 
	\begin{align*}
	{\lambda} &=-\frac{H''}{H}-(2n-2)\frac{F''}{F}+f''=\frac{H^2 q^2}{F^4}(2n-2)-\frac{H''}{H}-(2n-2)\frac{H'F'}{HF}+f'\frac{H'}{H}\\
	&=\frac{k}{F^2}-\frac{H^2 q^2}{F^4}2- \frac{F''}{F}-(2n-3)(\frac{F'}{F})^2-\frac{H'F'}{FH}+f'\frac{F'}{F}.
	\end{align*}
	If the metric is almost K\"{a}hler, then the system is simplified further
	\begin{align*}
		Hq &= FF'= \frac{q}{B} f',\\
		\lambda &= \frac{k}{F^2}-2n \big(\frac{F'}{F}\big)^2-2\frac{F''}{F}+f'\frac{F'}{F}.		
	\end{align*}
\end{lemma}
\begin{proof}
It follows from (\ref{reducedsolitionsystem}) and Lemma \ref{Rccomp}. Indeed, for $r=n-1$,
	\begin{align*}
	\lambda= \lambda g(N, N) &= \Rc(N, N)+ \He f(N, N)=-\frac{H''}{H}-(2r)\frac{F''}{F}+f''.
\end{align*}
Similarly, 
\begin{align*}	
	\lambda g(\partial_z, \partial_z)=H^2\lambda &= \Rc(\partial_z, \partial_z)+ \He f(\partial_z, \partial_z)=\Rc(\partial_z, \partial_z)+f' g(L\partial_z, \partial_z)\\
	&=H^2 \Big(\frac{H^2 q^2}{F^4}(2r)-\frac{H''}{H}-(2r)\frac{F'}{F}\frac{H'}{H}+f'\frac{H'}{H} \Big);\\
	\lambda g(E_i, E_i)= F^2\lambda &= \Rc(E_i, E_i)+ \He f(E_i, E_i)\\
	&=F^2 \Big(\frac{k}{F^2}-\frac{H^2 q^2}{F^4}2-\frac{F''}{F}-\frac{F'}{F}\frac{H'}{H}-(2r-1)\big(\frac{F'}{F}\big)^2+f'\frac{F'}{F} \Big).
	\end{align*}
	If the metric is almost K\"{a}hler, then $qH = FF'$. Since $\He f$ is $J$-invariant,
	\[f''=\He f(\nu, \nu)= \He f(\frac{\partial_z}{H},\frac{\partial_z}{H})= f'\frac{H'}{H}.\]
	Therefore, the system is reduced to 
	\begin{align*}
		\lambda &= -(2r+3)\frac{F''}{F}-\frac{F'''}{F'}+ f'',\\
		&= \frac{k}{F^2}-2n \big(\frac{F'}{F}\big)^2-2\frac{F''}{F}+f'\frac{F'}{F}.		
	\end{align*}
	Finally, one observes that $\ddt (\lambda F^2= k-2n F'^2-2F''F+ f'F'F)$ yields the first equation. 
\end{proof}

Thus, for a K\"{a}hler GRS constructed from ansatz \ref{ansatz} with $f=f(t)$, it is possible to solve the above system explicitly. Following \cite{DW11coho}, we consider the change of variables:
\begin{equation}
\label{transform1}
ds = FF' dt, ~~ \alpha(s):= H^2(t), ~~ \beta(s):=F^2(t), \varphi(s):= f(t).
\end{equation}
Using a dot to denote the derivative with respect to variable $s$: $\dot{X}=\partial_s X$, we have
\begin{align*}
	\partial_s \alpha=\dot{\alpha} &= 2H',  &\ddot{\alpha} &= \frac{2 {H''}}{H}\\
	\dot{\beta} &= \frac{2F{F'}}{H}, &\ddot{\beta}&=\frac{2{F'}^2+2F{F''}}{H^2}-\frac{2F{F'}{H'}}{H^3},\\
	\dot{\varphi} &= \frac{{f'}}{H}, &\ddot{\varphi}&= \frac{{f''}}{H^2}-\frac{{f'}{H'}}{H^3}. 
\end{align*}
\begin{lemma}
	\label{almostKahler}
	Let $(I\times P,  g, J)$ be given as in Ansatz \ref{ansatz} for $q\neq 0$. It is an almost K\"{a}hler GRS and $f=f(t)$ if and only if we have:
	\begin{align*}
	\beta(s) &=2s+A,\\ 
	\varphi(s) &=\frac{B}{q}s+C,\\
	\alpha(x) (2x+A)^{n-1} e^{-xB/q}\mid_{s_0}^s&= \frac{e^{sB/q}}{(2s+A)^{n-1}} \int_{s_0}^s (-2\lambda x+D)e^{-xB/q}(2x+A)^{n-1} dx.
	\end{align*}
\end{lemma}
\begin{proof} From (\ref{transform1}) and Lemma \ref{solitonsubmer1}, $\alpha$ satisfies a first order equation, for $\dot{X}=\frac{d}{ds} X$,
	\[{\lambda}(2s+A) ={k}-\dot{\alpha}-\frac{2(n-1)\alpha}{2s+A}+\frac{B}{q}\alpha.\]
\end{proof}

It order to obtain a global complete metric, one needs to smoothly extend the construction to singular orbits (if any). The following follows from the proof of \cite[Lemma 1.1]{EW00ivp}. We provide a direct proof as our ansatz (\ref{ansatz}) is fairly explicit.  

\begin{lemma}
	\label{smoothness}
	Let $I=(0, r)$ and $(I\times P, g)$ is given by ansatz \ref{ansatz} for $H(0)=0$, $F(0)>0$.  The metric can be extended smoothly to $t=0$ if and only if, for any integer $k$,
	\[H'(0)=1,~~ H^{(2k)}(0)=0=F^{(2k+1)}(0).\]
\end{lemma}
\begin{proof}
	We rewrite the metric in polar coordinates, for $x=t\cos(z)$ and $y=t\sin(z)$,
	\begin{align*}
		dt &= \frac{x}{t} dx +\frac{y}{t}dy,\\
		dz &= \frac{-y}{t^2} dx +\frac{x}{t^2}dy
	\end{align*}
	Then,
	\begin{align*}
		g &= dt^2+ H^2 (dz + q \alpha) \otimes (dz+ q\alpha)+ F^2 g_N\\
		&=t^{-2}(x^2 dx^2+y^2 dy^2+ xy dx\otimes dy+ xy dy\otimes dx)\\
		&+\frac{H^2}{t^4}(y^2 dx^2+x^2 dy^2- xy dx\otimes dy- xy dy\otimes dx)\\
		&+\frac{-qy H^2}{t^2}(\alpha \otimes dx + dx\otimes \alpha)+ \frac{qx H^2}{t^2}(\alpha \otimes dy+ dy\otimes \alpha)\\
		&+H^2 q^2 \alpha\otimes \alpha+F^2 g_N,\\
		&= dx^2 (\frac{H^2 y^2}{t^4}+\frac{x^2}{t^2})+ dy^2 (\frac{H^2 x^2}{t^4}+\frac{y^2}{t^2})+ (dx\otimes dy+ dy\otimes dx)(-\frac{H^2 xy}{t^4}+\frac{xy}{t^2})\\
		&+\frac{-qy H^2}{t^2}(\alpha \otimes dx + dx\otimes \alpha)+ \frac{qx H^2}{t^2}(\alpha \otimes dy+ dy\otimes \alpha)\\
		&+H^2 q^2 \alpha\otimes \alpha+F^2 g_N.
	\end{align*}
	Thus, the metric can be extended smoothly to $t=0$ if and only if the metric components 
	\[ \frac{y^2}{t^2}(\frac{H^2}{t^2}-1), \frac{x^2}{t^2}(\frac{H^2}{t^2}-1), \frac{xy}{t^2}(\frac{H^2}{t^2}-1), \frac{qy H^2}{t^2}, \frac{qx H^2}{t^2}, F^2 \]
	can be smoothly extended to $x=y=0$. According to \cite[Prop. 2.7]{KW74curv}, a function $\tilde{f}(x, y)=f(t, z)$ is smooth if and only if
	\begin{itemize}
		\item $f(t, z)=f(-t, z+\pi)$ for all $t, z$.
		\item $t^k (\frac{\partial^k f}{\partial t^k}(0, \theta))$ is a homogeneous polynomial of degree $k$ in $x$ and $y$, 
	\end{itemize} 
 	Applying such criteria to our case yields
 	\begin{itemize}
 		\item $H'(0)=1$ and $H^{(2n)}(0)=0$;
 		\item $F^{(2n+1)}(0)=0$.  
 	\end{itemize}
\end{proof}

Additionally, we will encounter the case that $(P, g_t)$ is an hyperbolic metric; accordingly, $(N, g_N)$ is then an Euclidean space. Thus, the following will be useful. 
\begin{lemma}
	\label{solitonsubmer2}
	Let $I\times P$ be equipped with the metric 
	\[dt^2+ g_t=dt^2+ e^{2A(t)z} \pi^\ast g_N+  H^2(t)\eta\otimes \eta\]
	for $\eta= dz$ and $\Rc_N =0$. The gradient Ricci soliton equation becomes 
	\begin{align*}
		A' &=0= H',\\
		{\lambda} &=-(\frac{A}{H})^2 (2m)=f''.
	\end{align*}
\end{lemma} 
\begin{proof} 
	We have
	\begin{align*} 
		2g_t L_t &= g'_t=2\frac{H'}{H}H^2 dz^2+ 2z A' e^{2Az}\pi^\ast g_N,\\
		L_t &= \frac{H'}{H}\partial_z \otimes dz+ zA'\id,\\
		L'_t &=(\frac{H''}{H}-(\frac{H'}{H})^2) \partial_z \otimes dz+zA''\id. 
	\end{align*}
	Consequently,
	\begin{align*}
		\tr{L_t} &=\frac{H'}{H}+zA'(2m),\\
		\tr{L^2_t} &=(\frac{H'}{H})^2+2m z^2 (A')^2,\\
		\tr{L'_t} &=(\frac{H''}{H}-(\frac{H'}{H})^2)+(2m)zA''.
	\end{align*}
	By the first equation of (\ref{reducedsolitionsystem}), one deduces that $A'=0$ or $A$ is constant. By the third equation of (\ref{reducedsolitionsystem}) and Lemma \ref{curvatrehyper}, $H$ is constant. The result then follows. 
	
\end{proof}
\section{Induced Symmetry}
\label{symm}
Let $(M, g, J)$ be an almost Hermitian manifold with a smooth non-trivial function $f: M\mapsto \mathbb{R}$. We'll examine how the symmetry of $(M, g, J)$ induces certain symmetry on level sets of function $f$. Notably, throughout this section, $f$ is not necessarily the potential function to a soliton structure.

For each $c\in f(M)$, $M_c:=f^{-1}(c)$ is called a level set of $f$. By the regular level set theorem \cite{tubook11}, if $c$ is a regular value, then $M_c$ is a smooth submanifold of co-dimension one. 
As $V=\frac{\nabla f}{|\nabla f|}$ is well-defined along $M_c$, let $\zeta=-J(V)$ and $\eta$ be the dual $1$-form to $\zeta$ with respect to the induced metric on $M_c$. 
We define $\Phi$ on $TM_c$ by
\[\Phi X+\eta(X) V= JX. \]
\begin{proposition}
	\label{levelalmost}
	Let $(M, g, J)$ be an almost Hermitian manifold with a smooth non-trivial function $f: M\mapsto \mathbb{R}$. If $c$ is a regular value of $f$, then $(M_c, g, \zeta, \eta, \Phi)$ is an almost contact metric structure.
\end{proposition}
\begin{proof}
	If $X=a\zeta+X_1\in TM_c$ for $X_1$ a section of $TM_c$ and $X_1\perp \zeta$, then it is immediate that $\Phi(X)=JX_1$ is also a section of $TM_c$. We check
	\begin{align*}
		\Phi^2(X) &=\Phi(J(X_1))=J(J(X_1))-\eta(J(X_1))V,\\
		&=-X_1-g(\zeta, J(X_1))V=-X_1+g(J(\zeta), X_1)V=-X_1+g(V, X_1)V,\\
		&=-X_1=-X+a\zeta=-X+\eta(X)\zeta.
	\end{align*}
	Therefore, $\Phi^2=-\text{Id}+\zeta\otimes \eta$ and $(M_c, \Phi, \zeta, \eta)$ is an almost contact structure. 
	Next, for $X=a\zeta+X_1$ and $Y=b\zeta+Y_1$, we compute
	\begin{align*}
	g(\Phi X, \Phi Y) &= g(J(X_1), J(Y_1))=g(X_1, Y_1)= g(X-a\zeta, Y-b\zeta)\\
	&=g(X, Y)-a g(\zeta, Y)-b g(X,\zeta)+ab g(\zeta, \zeta)\\
	&=g(X, Y)-2ab+ab=g(X, Y)-\zeta(X)\zeta(Y). 
	\end{align*}
	Thus, $(M_c, \Phi, \zeta, \eta)$ is an almost contact metric structure. 
\end{proof}

We will collect useful observations. 
\begin{lemma}
	\label{lie0}
	Suppose that $\lie_X g=0$. 
	\begin{enumerate} [label=(\roman*)]
		\item  $\lie_X f=\text{constant} \iff \lie_X \nabla f=0$;
		\item $\lie_X \nabla f\perp \nabla f \iff \lie_X |\nabla f|^2=0$. 
		\item Let $\gamma$ be the 1-form dual to a vector field $Z$. $\lie_X Z=0 \iff \lie_X \gamma=0$.
	\end{enumerate}
\end{lemma}
\begin{proof}
We compute
	\begin{align*}
	g(\lie_X \nabla f, Y) &= g(\nabla_X \nabla f-\nabla_{\nabla f}X, Y)\\
	&=\He f(X, Y)+g(\nabla_Y X, \nabla f)-\lie_X g(Y, \nabla f)\\
	&=\He f(X, Y)-g(X, \nabla_Y \nabla f)+Y(\lie_X f)-\lie_X g(Y, \nabla f)\\
	&=Y(\lie_X f)-\lie_X g(Y, \nabla f).
	\end{align*}
Since  $\lie_X g=0$ and $X$ and $Y$ are arbitrary, $\lie_X f=\text{constant} \iff \lie_X \nabla f=0$. For the second statement, 
	\begin{align*}
		\lie_X|\nabla f|^2 &= 2g(\nabla_X \nabla f, \nabla f)\\
		&= 2g([X,\nabla f], \nabla f)-2g(\nabla_{\nabla f}X, \nabla f)\\
		&= 2g([X,\nabla f], \nabla f)-(\lie_X g)(\nabla f, \nabla f).   
	\end{align*}
Finally, 
\begin{align*}
	(\lie_X \gamma)Y &= X(\gamma Y)-\gamma (\lie_X Y)\\
	&= X g(Z, Y)-g(Z, [X, Y])= g(\nabla_X Z, Y)+ g(Z, \nabla_Y X)\\
	&= g([X, Z], Y)+\lie_X g(Y, Z).
\end{align*}
The conclusion follows.   
\end{proof}

\begin{lemma} \label{lie4}
	Suppose that $\lie_X V=0$ and $\lie_X\eta=0$. $\lie_X \Phi=0$ if and only if $\lie_X J=0$. 
\end{lemma} 
\begin{proof} We compute
	\begin{align*}
		(\lie_X \Phi)Y &=[X, \Phi(Y)]-\Phi([X, Y])\\
		&=[X, J(Y)-\eta(Y)V]-J([X, Y])+\eta([X, Y])V\\
		&= (\lie_X J)Y-[X, \eta(Y)V]+\eta([X, Y])V\\
		&= (\lie_X J)Y-\eta(Y)[X, V]-\nabla_X(\eta(Y))V+\eta([X, Y])V\\
		&=(\lie_X J)Y-\eta(Y)\lie_X V-(\lie_X\eta)(Y) V
	\end{align*}
\end{proof}

\begin{proposition}
	\label{restrictionsymmetry}
If $\lie_X g=\lie_X J=\lie_X f=0$ then we have, along each regular level set,
	\begin{align*}
		\lie_X \zeta &=0, &\lie_X \eta &=0,	&\lie_X \Phi &=0. 
	\end{align*}
\end{proposition}

\begin{proof}
	Recall that, along a regular level set, $\zeta=-J(\frac{\nabla f}{|\nabla f|})$. By Lemma \ref{lie0}, the assumption implies
	\[\lie_X \frac{\nabla f}{|\nabla f|}=0.\]
	Since $\lie_X J=0$, for any vector field $Y$,
	\[\nabla_{JY}X= J\nabla_Y X.\]
	Thus, $\lie_X (JV)=J \lie_X V $	and we obtain $\lie_X \zeta =0$. The others follow from Lemmas \ref{lie0} (iii) and Lemma \ref{lie4} respectively.
\end{proof}

\begin{proposition}
	\label{restrictiontrivial}
	Suppose that $\lie_X g=0$ and $\lie_X \nabla f=0$. If $X_{\mid_{M_c}}\equiv 0$ then $X\equiv 0$.
\end{proposition}
\begin{proof}
	We compute
	\begin{align*}
		\nabla_{\nabla f} X &= -[X, \nabla f]+\nabla_X \nabla f\\
		&= \He{f}(X, \cdot)=0.
	\end{align*}
	The last equality holds since $X_{\mid_{M_c}}\equiv 0$.	Since a Killing vector field is completely determined by its zero and first order values at a point \cite[Chapter VI]{KNvolumeI96}, $X$ must be trivial. 
\end{proof}

\begin{theorem}
	\label{inducegroup}
	Let $(M, g, J)$ be an almost Hermitian manifold with a smooth non-trival function $f: M\mapsto \mathbb{R}$ and $G$ be a group of symmetry preserving $g, J,$ and $f$. Then $G$ is also a group of symmetry for $(M_c, g, \zeta, \eta, \Phi)$ as an almost contact metric structure.    
\end{theorem}
\begin{proof}
	
Let $u: M\mapsto M$ be an isometry preserving $J$ and $f$. As $f(u(a))=f(a)$, $u(M_c)=M_c$ and $u$ induces a map $u_c: M_c\mapsto M_c$. The proof will follow from the following claims. 

\textit{Claim:} If $u_c$ is an identity map then so is $u$.

\textit{Proof:} An isometry is necessarily an affine transformation which preserves parallelism \cite[Chapter VI]{KNvolumeI96}. The result then follows.

\textit{Claim:} $\frac{\nabla f}{|\nabla f|}$ is $u$-invariant. Consequently, so is $\zeta= -J(\frac{\nabla f}{|\nabla f|})$. 

 \textit{Proof:} Since $f$ is $u$-invariant, so is $df$. As $\nabla f$ is the dual of $df$ via $g$ and each is $u$-invariant, the first statement follows. The second is because $J$ is $u$-invariant.
 
 \textit{Claim:} $\eta$ and $\Phi$ are $u$-invariant. 
 
 \textit{Proof:} Because $\eta$ is the dual of an $u$-invariant vector field and $u$ is an isometry, $\eta$ is $u$-invariant. Next, one recalls 
 \[\Phi(\cdot)= J(\cdot)-\eta(\cdot)\frac{\nabla f}{|\nabla f|}\]
 and each component is $u$-invariant. The result then follows. 
\end{proof}

Immediately, there is an estimate on the size of the symmetry group. 
 \begin{corollary}
 	\label{maxdim}
 		Let $(M, g, J)$ be an almost Hermitian manifold with a smooth non-trival function $f: M\mapsto \mathbb{R}$. The dimension of the group of symmetry is at most $n^2$. Equality happens if each connected component of a regular level set of $f$ is a model space given in Theorem \ref{tannoclassificationalmost}. 
 \end{corollary}
 \begin{proof}
 	It follows from Theorem \ref{inducegroup} and the corresponding result for an almost contact metric structure, Theorem \ref{tannoclassificationalmost}. 
 	
 \end{proof}

Under the setup of Ansatz \ref{ansatz}, there is a converse statement. Let $(M, g, J)$ be an almost Hermitian manifold such that over a dense subset, the metric is given by the Ansatz \ref{ansatz}. Let $X$ be an infinitesimal automorphism vector field on $N$. That is, 
\[\lie_X g_N= \lie_X J_N=0=\lie_X \omega_N.\]
Let $X^\ast$ be its horizontal lift to $(P, g_t)$. By Cartan's formula,
\[3d\omega (W, Y, Z)= (\lie_W\omega)(Y, Z)- d(\mathfrak{i}_W\omega)(Y, Z).\]
For $\omega=\pi^\ast(\omega_N)$, $W=X^\ast$, $d\omega=0=\lie_{X^\ast} \omega$. Thus, $\mathfrak{i}_{X^\ast} \omega$ is closed. If $P$ is simply connected then there exists a function $\ell$ such that 
\[ \mathfrak{i}_{X^\ast} \pi^\ast(\omega)= d\ell.\]

\begin{lemma}
	\label{symmetryup}
	If $P$ is simply connected then the vector field $X^\ast- \ell\partial_z$ is independent of $t$ and is an infinitesimal symmetry of $(P, g_t, \eta, \zeta, \Phi)$.	
\end{lemma}
\begin{proof}
	$X^\ast$ is independent of $t$ as it only depends on the submersion $\pi$ and the fixed subspace which is $g_t$ perpendicular to $\partial_z$ for all $t$. $\ell$ is independent of $t$ as the proof of the Poincare's lemma is topological. The rest is straightforward; see \cite[Lemma 5.1]{tanno69} for details.  
\end{proof}
\begin{remark}
If $N$ is simply-connected and $\pi$ is trivial, one can choose $\ell$ to be constant on each fiber.
\end{remark}

Since $(P, g_t)$ is complete for each $t$, it is possible to construct $G$, the group consisting of automorphisms generated by vector fields of the form $X^\ast -\ell \partial_z$ in Lemma \ref{symmetryup} and the Killing vector field $\partial_z$.     

\begin{proposition} 
	\label{conversesymmetry}
If either $P$ is simply connected or $N$ is simply connected and $P$ is a trivial bundle over $N$ then $G$ is a group of automorphism for $(M, g, J)$.
\end{proposition}
\begin{proof}
	Let $X_i$ be either a vector field from Lemma \ref{symmetryup} or $\partial_z$. Then, immediately,
	\[\lie_{X_i} g= \lie_{X_i}(dt^2+g_t)=\lie_{X_i}g_t=0. \] 
	For $V=\partial_t$, by Lemmas \ref{lie4} and \ref{symmetryup},  $\lie_X J=0$. The result then follows. 
\end{proof}

\section{Rigidity of the Maximal Dimension}
\label{rigid}

We are now ready to prove main results.

\begin{proof}[Proof of Theorem \ref{main1}] First, by Corollary \ref{maxdim}, $\text{dim}(G)\leq n^2$ and equality happens only if each connected component of a regular level set must be one of the model spaces, described in Theorem \ref{tannoclassificationalmost}, which is a total space of a line or circle bundle over $\mathbb{N}(k)$ with the fiber projection $\pi$.
	
	
Furthermore, as $G$ acts transitively on each such component, each regular level set is a principal orbit and the orbit space of $G$-actions on $M$ is of dimension one. Thus, $(M, g, J, f)$ is of cohomogeneity one. By Sard's theorem, the set of singular values for $f:M \mapsto \mathbb{R}$ is of measure zero in $\mathbb{R}$. Thus, by continuity, nearby connected components must be obtained from the same model $P$. Locally, the metric can be written as,     
\[g= dt^2 +g_t,~~ f=f(t).\]
 
Next, we consider cases as described in Theorem \ref{tannoclassificationalmost}.

\textbf{Case 1:} Each connected component $(P, g_t, \zeta_t, \eta_t, \Phi_t)$ is a deformation of a homogeneous Sasakian metric with constant $\Phi$-sectional curvature. Thus, $g_t$ is obtained from a standard Sasakian metric via an $a$-homothety and a $\pm b$ deformation. By fixing a background $\eta$ and $\zeta=\partial_z$ on each fiber, Theorem \ref{tannoSasa} and Lemma \ref{curvaturedeform} imply that, for $F^2 = ab$ and $H =a$,  
\begin{align*}
	g_t &= F^2(t)\pi^\ast g_{\mathbb{N}}+ H^2(t)\eta\otimes \eta,\\
	\eta (\partial_z) &= 1,~~ d\eta =\pi^\ast\omega_N.
\end{align*}

\textbf{Case 2:} $P$ is a trivial bundle over $\mathbb{N}(k)$. Thus,
\begin{align*}
g_t &= H^2(t)dz^2+ F(t)^2 \pi^{\ast} g_{\mathbb{N}}.
\end{align*}

\textbf{Case 3:} $(P, g_t)$ is a hyperbolic metric. That is, 
\begin{align*}
g_t &= H^2 dz^2+ e^{2A(t)z} g_N.
\end{align*} 	
One direction then follows from Lemmas \ref{solitonsubmer1} (for case 1 and case 2 corresponding to $q=1$ and $q=0$) and Lemma \ref{solitonsubmer2} (for case 3). \\

For the reverse direction, if the soliton is locally constructed by the ansatz \ref{ansatz} and $P$ is simply connected, then its automorphism group is the same as the group of symmetry of $(P, g_t, \eta_t, \zeta_t, \Phi_t)$ for each regular value $t$ by Prop. \ref{conversesymmetry}. For $N=\mathbb{N}(k)$ such group is of dimension $n^2$ by \cite{tanno69}. The hyperbolic case case is trival as the metric is a product. 
 
\end{proof}

The next results will pave the way to the proof of Theorems \ref{main0}.  
\begin{proposition}
	\label{main2}
	Let $(M^{2n}, g, J, f, \lambda)$ be a non-trivial almost K\"{a}hler GRS with $G$ the group of symmetry. If $\text{dim}(G)=n^2$ then either 
	\begin{enumerate} [label=(\roman*)]
		\item the soliton belongs to case (i) of Theorem \ref{main1} with $q=1$, $f$ monotonic, and each level set is connected. 
		\item the soliton splits as $(M_1, g_, J_1, f_1, \lambda)\times (\mathbb{N}, g_{\mathbb{N}}, J_{\mathbb{N}}, f_{\mathbb{N}}, \lambda)$ for  $(M_1, g_, J_1, f_1, \lambda)$ a K\"{a}hler GRS in real dimension two. 
	\end{enumerate}
	
\end{proposition}
\begin{proof} We continue from the proof of Theorem \ref{main1}. For the first case $q=1$, the result follows from Lemma \ref{almostKahler}. For the case $q=0$, by Lemma \ref{ricciansatz}, $F$ is constant. Thus, the soliton must splits as a Riemannian product
	\[ M_1 \times M_2.\]
By \cite[Lemma 2.1]{PW09grsym} and the discussion after Lemma \ref{solitonsubmer1}, each $(M_i, g_i, J_i. f_i, \lambda)$ is a K\"{a}hler GRS. As $(M_2, g_2, J_2)= (\mathbb{N}, g_{\mathbb{N}}, J_{\mathbb{N}})$ the result follows. Finally, for the hyperbolic case, the product metric is not almost K\"{a}hler. 


\end{proof}

For Theorem \ref{main0}, we will divide the proof into two parts corresponding roughly to $\lambda\neq 0$ and $\lambda=0$. 

\begin{proof}[Proof of Theorem \ref{main0} (part I)] Let $G$ be the largest connected group of isometries. 
	
\textit{Claim 1:} $G$ preserves $J$. 
	
\textit{Proof:}	By A. Lichnerowicz \cite{lich54}, for an irreducible K\"{a}hler manifold, $G$ preserves $J$ if $n$ is odd or if $n$ is even and $\Rc$ does not vanish. As the soliton is non-trivial, $\Rc\neq 0$. \\

Next, we give the proof in case $\lambda \neq 0$.

\textit{Claim 2:} If $\lambda\neq 0$, the group of isometries preserves $f$.
	
\textit{Proof:}	By \cite{PW09grsym}, in case of $\lambda\neq 0$, for a Killing vector field $X$, either $\lie_X f=\nabla_X f\equiv 0$ or $\nabla X\equiv 0$ and the manifold splits off a line or a circle. Since the metric is K\"{a}hler, 
	\[\nabla X\equiv 0 \rightleftarrows \nabla(JX)\equiv 0.\]
	Thus, it splits off a line/circle if and only if there is a decomposition with a flat factor with respect to the K\"{a}hler structure, which contradicts the irreducibility.\\ 

 The claims imply that $G$ is contained in the group of symmetry preserving $g, J,$ and $f$. By Theorem \ref{main1}, the dimension of $G$ is at most $n^2$ and equality implies that the metric must be constructed from ansatz \ref{ansatz} due to Prop. \ref{main2}(i). \\


		\textit {Claim 3:} If $\lambda> 0$ then $\mathbb{N}=\mathbb{CP}^{n-1}$.  
	
	\textit{Proof:} Prop. \ref{main2} gives explicit solution to the system of ODE. Since $\alpha(s)$ and $\beta(s)$ do not both approach $\infty$ as $s\rightarrow -A/2$, the metric is only complete if there is a singular orbit and one needs to smoothly compactify such an end. By taking a scaling if necessary, one assumes that the singular orbit is at $s=0=t$ and the metric is defined in a neighborhood where $s>0$. Thus, immediately, $\beta\geq 0$ if and only if $A\geq 0$.\\
	
	Assume that $\mathbb{N}\neq \mathbb{CP}^{n-1}$ then it is non-compact. Then, at $t=s=0$, one can only collapse the fiber; that is, $\alpha(0)=H(0)=0, \beta(0)=F^2(0)>0$. By Prop. \ref{smoothness}, the smoothness of the metric requires 
	\begin{align*}
		\frac{\partial \alpha}{\partial s}(0) &= 2\frac{\partial H}{\partial t}(0)= 2. 
	\end{align*}
	Evaluating the equation 
	${\lambda}(2s+A) ={k}-\dot{\alpha}-\frac{2(m-1+q^2)\alpha}{2s+A}+B\alpha$ at $s=0$ yields
	\[k-\lambda A=D= 2. \]
	For $\lambda > 0$, it implies that $k>0$, a contradiction to  $\mathbb{N}(k)\neq \mathbb{CP}^{n-1}$.\\
	
	Finally, the reverse direction follows from Claim 1, Prop. \ref{main2}, and Theorem \ref{main1}. 
	
	\end{proof} 
\begin{remark}
	By \cite{PW09grsym}, for a Killing vector field $X$, the gradient of $Xf$ is parallel. Thus, either the soliton locally splits a Riemannian factor or $Xf$ is constant. Because of our irreducibility assumption, $Xf=\text{constant}$. If $\lambda\neq 0$, the soliton identity ($\SS+|\nabla f|^2-2\lambda f= \text{constant}$) implies that $f$ must assume all real values, a contradiction to a lower bound on $\SS$ as shown in \cite{zhang09completeness}. There is also a more general argument according to \cite{wyliepersonal}. 
\end{remark}

\begin{proof}[Proof of Theorem \ref{main0} (part II)] Here we consider the case $\lambda =0$. Without loss of generality, we assume the scalar curvature function $\SS$ is non-constant (otherwise, by the maximum principle, $\Rc=0$, a contradiction). By Claim 1, $G$ is contained in the group of symmetry preserving $g, J,$ and $\SS$. By Corollary \ref{maxdim}, $\text{dim}(G)\leq n^2$ and equality implies that each connected component of a regular level, denoted by $P$, is a model space described in Theorem \ref{tannoclassificationalmost}. We argue as in Prop \ref{main2} to deduce that $P$ is a deformed standard Sasakian structure. Thus, locally, the metric must be constructed from ansatz \ref{ansatz} with $N=\mathbb{N}(k)$ and $q=1$. 
	
\textit {Claim:} If $\lambda=0$ then $\mathbb{N}=\mathbb{CP}^{n-1}$. 	

\textit{Proof:} If $f$ is $G$-invariant then we can repeat the argument as in the case $\lambda>0$. If not, then, there is a neighborhood of a point such that
\[\nabla f= f_0 \partial_t + X,\]
for non-trivial $X \perp \partial_t$. By Cor. \ref{maxdim} and Prop. \ref{conversesymmetry}, any Killing vector field generated by $G$ is perpendicular to $\partial_t$. Since $J(\nabla f)$ is a Killing, \[X\perp \text{span}\{\partial_t, \partial_z\}.\] 
By Lemma \ref{Rccomp}, the Ricci curvature is diagonalized and $\nabla \SS$ is an eigenvector. Thus, the equation $\Rc (\nabla f)= \frac{1}{2}\nabla \SS$ implies $\Rc(X, X)=0$. Applying Lemma \ref{Rccomp} again yields
\[k-2n (F')^2-2F'' F=0.\]
If $k\leq 0$ then immediately, $F''\leq 0$ and $F'=\frac{H}{F}$ is positive and non-increasing. Since $\mathbb{N}(k)$ is non-compact, one concludes that there is no singular orbit corresponding to $H\rightarrow 0$. However, as $t\rightarrow -\infty$, since $F'$ is positive and non-increasing, there must be a $t_0$ such that $F(t_0)=0$, a contradiction. Therefore, $k>0$ and $\mathbb{N}=\mathbb{CP}^{n-1}$. \\

By Cor. \ref{maxdim} and Prop. \ref{conversesymmetry}, $G$ is constructed from the automorphism group of $\mathbb{N}(k)$. Since $\mathbb{N}=\mathbb{CP}^{n-1}$, $G$ is compact and, consequently, $f$ can be modified to be $G$-invariant. 
\end{proof}

Now we prove Corollaries \ref{main3} and \ref{nondegenrate}.

\begin{proof}[Proof Corollary \ref{main3}]
If it is irreducible, Theorem \ref{main0} applies. Otherwise, we consider the De Rham decomposition of a K\"{a}hler manifold \cite[Theorem IX. 8.1]{KNvolumeII96}:
\[(M^n, g, J)= \Pi_{i=0}^k (M_i^{n_i}, g_i, J_i),\]
with $\sum_{i} n_i=n$ (complex dimensions). Here $(M_0, g_0, J_i)$ is Euclidean (if there is no such factor, we interpret $M_0$ as a point) and each $(M_i, g_i, J_i)$ is an irreducible K\"{a}hler manifold. By \cite[Lemma 2.1]{PW09grsym}, the soliton structure splits accordingly. That is,  \[(M^n, g, f, J, \lambda)= \Pi_{i=0}^k (M_i^{n_i}, g_i, f_i, J_i, \lambda).\]

By \cite[Theorem 1]{hano55} (see also \cite[Theorem VI.3.5]{KNvolumeI96}), the identity component of the isometry group of $(M, g)$ also decomposes as a product. In particular, each isometric transformation $u$ can be written as $(u_0,..., u_k)$ where each $u_i$ is an isometric transformation for $(M_i, g_i)$. Thus, it is immediate that $u$ preserves $J$ if and only if each $u_i$ does $J_i$. Therefore, the group of automorphisms-preserving both $g$ and $J$- of $(M, g, J)$ also splits as a product.     

If $(M_i, g_i, f_i, J_i, \lambda)$ is trivial then, by Tanno's theorem \cite[Theorem A]{tanno69Hermitian}, its group of automorphisms is of dimension at most $n_i(n_i+2)$. Suppose there are $r$ trivial factors. For the rest, by Theorem \ref{main0}, the group of isometry of the $j$-factor is of dimension at most $n_j^2$. Thus, the dimension of a connected component of the product group is at most,
	\[ \sum_{i=0}^r n_i(n_i+2)+ \sum_{j=r+1}^k n_j^2 \leq A(A+2)+B^2,\]
	for $A=\sum_{i=0}^r n_i$ and $B=\sum_{j=r+1}^k n_j$. As the soliton is non-trivial $A<n$. Thus, 
	\[A(A+2)+(n-A)^2= n^2-2A(n-1-A)\leq n^2.\]
	Equality happens if and only if either $A=n-1$ and there is exactly one Euclidean or K\"{a}hler-Einstein factor of complex dimension $(n-1)$ with maximal symmetry or $A=0$ (irreducible). In the former case, the other factor is a complete non-trivial K\"{a}hler GRS in complex dimension one. For $\lambda>0$, there is no non-trivial complete example in that dimension \cite{BM15}. The result then follows.
	


\end{proof}

 \begin{proof}[Proof of Corollary \ref{nondegenrate}]
 	By \cite{KNauto57}, the largest connected group of affine transformations $G$ contains of automorphisms. Non-degeneracy implies the Riemannian metric does not split off any flat factors. Thus, $f$ is $G$-invariant. The result then follows from Theorem \ref{main1} and Corollary \ref{main3}. For the reverse direction, the non-degeneracy follows as the Ricci curvature is non-singular at some point \cite{KNauto57}.  
 \end{proof}

\section{Appendix}
\subsection{Convention}
Here are our conventions:
\begin{itemize}
	
	\item $\lie$ denotes the Lie derivative. 
	\item The convention of exterior derivative, for an $m$-form $\alpha$,
	\begin{align*}(m+1)d\alpha(X_0,....X_m) &= \sum_{i}(-1)^i X_i (\alpha(X_0,....,\hat{X_i},...X_m))\\
		&+\sum_{i<j}(-1)^{i+j}\alpha([X_i, X_j], X_0,..., \hat{X_i},...,\hat{X_j},...X_m).\end{align*}
	Consequently,
	\[(dx^{i_1}\wedge... \wedge dx^{i_k})(\partial_{x_i},... \partial_{x_k})=\frac{1}{k!}. \]
	\begin{remark}
		Our convention agrees with \cite{chowluni} but differs from \cite{besse} by a scaling. 
	\end{remark}
\item 	The interior product for a $m$-form is defined as
\[(\mathfrak{i}_X \alpha (Y_1,..., Y_m))=m \alpha(X, Y_1,..., Y_m).\]
The following identity is the so-called Cartan's formula for a differential form
\[\lie_X \alpha=(d \circ \mathfrak{i}_X+\mathfrak{i}_X\circ d)\alpha.\]

	\item On a Riemannian manifold $(M, g)$, there is a unique Levi-Civita connection $\nabla: TM\times C^\infty(TM)\mapsto C^\infty(TM)$. The connection induces a Riemannian curvature via the covariant second derivative: 
	\[\nabla^2_{X, Y}= \nabla_X\nabla_Y-\nabla_{\nabla_X Y}.\]
	The Riemann curvature $(3,1)$ tensor and $(4,0)$ tensor are defined as follows,
	\begin{align*}
		\RR(X,Y,Z)&=\nabla^2_{X,Y}Z -\nabla^2_{Y,X}Z\\
		\RR(X,Y,Z,W)&=g(\nabla^2_{X,Y}Z -\nabla^2_{Y,X}Z, W).
	\end{align*}

	\begin{remark}
		Our sign convention agrees with \cite{chowluni, BGbookSasakian08}. Our $(3,1)$ curvature tensor differs from one of \cite{besse} by a sign. 
	\end{remark}
	Furthermore, the curvature can be seen as an operator on the space of two forms. For an orthonormal basis $\{e_i\}_i$ and any $2$-form $\alpha$, \[\RR(\alpha)(e_i, e_j)=\sum_{k<l}\RR(e_i, e_j, e_k, e_l)\alpha(e_k, e_l).\]
	Consequently, due to our exterior derivative convention,  
		\[\RR(X\wedge Y) (Z, W) =\frac{1}{2}\RR(X, Y, Z, W).\]
	The sectional curvature and Ricci curvature are defined as follows, 
	\begin{align*} K(X, Y) &=\RR(X, Y, Y, X),\\ 
		\Rc_{ik} &=\sum_j \RR_{ijjk}.
	\end{align*}
	
	
\end{itemize}
\subsection{Submersion}
A differentiable map between smooth manifolds $\pi: P\mapsto N$ is a submersion if the pushforward of the tangent space at each point is surjective. That is, for $p\in P$, $\pi_\ast (T_xP)=T_{\pi(p)}N$. A Riemannian submersion is a submersion between Riemannian manifolds such that the differential above is a linear isometry. 

We consider a Riemannian submersion $\pi: (P^{2n+1}, g_p)\mapsto (N^{2n}, g_N)$ such that each fiber is a geodesic line or circle with tangential vector field $\zeta$ such that $g(\zeta, \zeta)=1$. The submersion naturally decomposes $TP$ into vertical and horizontal distributions. Let $(\cdot)^{\mathcal{H}}$ and $(\cdot)^{\mathcal{V}}$ denote the horizontal and vertical parts, respectively, of a vector field on $P$. The vertical subspace consists of multiples of $\zeta$. Furthermore, for each vector field on $N$ there is a unique horizontal vector field on $P$ such that they are $\pi$-related. We'll collect useful lemmas whose proofs can be found in \cite{pe06book, besse, BGbookSasakian08} or a straightforward calculation.  

\begin{lemma} For horizontal vector fields ${X}$ and ${Y}$:  
	\begin{enumerate} [label=\roman*)]
		\item $[\zeta, {X}]$ is vertical,
		\item $g([{X}, {Y}], \zeta) =2g(\nabla_{{X}}{Y}, \zeta)= 2g(\nabla_{{Y}}\zeta, {X})=-2g(\nabla_{\zeta}{X}, {Y}),$
		\item $\zeta$ is a Killing vector field. 
	\end{enumerate}
	
\end{lemma}

The curvature of a submersion can be computed via B. O'Neill's $A$ and $T$ tensors \cite{ONeill66}. Since each fiber is totally geodesic, $T\equiv 0$ and only $A$ is non-trivial.  One recalls 
\begin{align*}
	A_X \zeta = (\nabla_X \zeta)^{\mathcal{H}}, &~~A_X Y = (\nabla_X Y)^{\mathcal{V}}.
\end{align*} 
For horizontal vector fields $X, T, Z, W$,
\begin{align*}
	\RR_P({X}, {Y}, {Z}, {W}) &= \RR_{N}(X, Y, Z, W) +2g(A_X Y, A_Z W)+g(A_X Z, A_Y W)-g(A_X W, A_Y Z),\\
	\RR_P({X}, \zeta, {Z}, \zeta) &= -g(A_X \zeta, A_Z \zeta).
\end{align*}
\begin{remark}
	These formulas differ from ones of \cite[Chapter 9]{besse} by a sign convention. 
\end{remark}

Consequently, there are corresponding identities for the sectional curvature and Ricci curvature. For orthonormal horizontal vectors $X, Y$
\begin{align*}
	K_P(X, Y) = K_{N}(X, Y)-3 |A_X Y|^2; &~~K_P(X, \zeta) = |A_X \zeta|^2, \\
	\Rc_P(X, Y) = \Rc_{N}(X, Y)-2g(A_X, A_Y);~~\Rc(X, \zeta) = 0; &~~ \Rc(\zeta, \zeta) = g(A\zeta, A\zeta).
\end{align*}
Here, as $\{E_i\}_{i=1}^{2n}$ denotes a local orthonormal frame for the horizontal distribution, 
\begin{align*}
	g(A_X, A_Y) &= \sum_i g(A_X E_i, A_Y E_i)= g(A_X \zeta, A_Y \zeta),\\
	g(A\zeta, A\zeta) &=\sum_i g(A_{E_i}\zeta, A_{E_i}\zeta).
\end{align*}

Next, we restrict to the submersion given by ansatz \ref{ansatz}. It is immediate that $\zeta= \frac{1}{H} \partial_z$ is a Killing vector field. 
In this situation, tensor $A$ can be computed immediately. 
\begin{lemma} \label{computeA}
	For horizontal vector fields $X$ and $Y$
	\begin{align*}
		2 H d\eta(X, Y) &= -g([{X}, {Y}], \zeta)= \frac{Hq}{F^2} g (X, JY),\\
		A_X Y &=-\frac{Hq}{F^2}g(X, JY)\zeta\\
		A_X\zeta &=-\frac{Hq}{F^2} JX.
	\end{align*}
\end{lemma}

\begin{lemma}
	\label{submersioncurvature}
	The sectional and Ricci curvature of $(P, g)$ are given by, for orthonormal horizontal vectors $X, Y$  
	\begin{align*}
		K(X, Y) &= \frac{1}{F^2}K_N(FX, FY)-3\frac{H^2 q^2}{F^4} g(X, JY)^2,\\
		K(X, \zeta) &= \frac{H^2 q^2}{F^4},\\
		\Rc(X, Y) &= \Rc_N (X, Y)-2\frac{H^2q^2}{F^4}g(X, Y)\\\
		\Rc(\zeta, \zeta) &=\frac{H^2q^2}{F^4}2m.
	\end{align*}
\end{lemma}

\def\cprime{$'$}
\bibliographystyle{plain}
\bibliography{bioMorse}

\def\cprime{$'$}
\begin{thebibliography}{10}

\bibitem{BCCD22KahlerRicci}
Richard~H Bamler, Charles Cifarelli, Ronan~J Conlon, and Alix Deruelle.
\newblock A new complete two-dimensional shrinking gradient
  {K}\"{a}hler-{R}icci soliton.
\newblock {\em Geom. Funct. Anal.}, 34:377--392, 2024.

\bibitem{BM15}
Jacob Bernstein and Thomas Mettler.
\newblock Two-dimensional gradient {R}icci solitons revisited.
\newblock {\em Int. Math. Res. Not. IMRN}, (1):78--98, 2015.

\bibitem{besse}
Arthur~L. Besse.
\newblock {\em Einstein manifolds}, volume~10 of {\em Ergebnisse der Mathematik
  und ihrer Grenzgebiete (3) [Results in Mathematics and Related Areas (3)]}.
\newblock Springer-Verlag, Berlin, 1987.

\bibitem{bohmwilking}
Christoph B{\"o}hm and Burkhard Wilking.
\newblock Manifolds with positive curvature operators are space forms.
\newblock {\em Ann. of Math. (2)}, 167(3):1079--1097, 2008.

\bibitem{BGbookSasakian08}
Charles~P. Boyer and Krzysztof Galicki.
\newblock {\em Sasakian geometry}.
\newblock Oxford Mathematical Monographs. Oxford University Press, Oxford,
  2008.

\bibitem{b12rot}
Simon Brendle.
\newblock Rotational symmetry of self-similar solutions to the {R}icci flow.
\newblock {\em Invent. Math.}, pages 1--34, 2012.

\bibitem{brendle14rotahigh}
Simon Brendle.
\newblock Rotational symmetry of {R}icci solitons in higher dimensions.
\newblock {\em J. Differential Geom.}, 97(2):191--214, 2014.

\bibitem{bs091}
Simon Brendle and Richard Schoen.
\newblock Manifolds with {$1/4$}-pinched curvature are space forms.
\newblock {\em J. Amer. Math. Soc.}, 22(1):287--307, 2009.

\bibitem{bs072}
Simon Brendle and Richard~M. Schoen.
\newblock Classification of manifolds with weakly {$1/4$}-pinched curvatures.
\newblock {\em Acta Math.}, 200(1):1--13, 2008.

\bibitem{BDW15}
M.~Buzano, A.~S. Dancer, and M.~Wang.
\newblock A family of steady {R}icci solitons and {R}icci flat metrics.
\newblock {\em Comm. Anal. Geom.}, 23(3):611--638, 2015.

\bibitem{CS2018classification}
Jacob Cable and Hendrik S{\"u}{\ss}.
\newblock On the classification of {K}{\"a}hler--{R}icci solitons on gorenstein
  del pezzo surfaces.
\newblock {\em European Journal of Mathematics}, 4(1):137--161, 2018.

\bibitem{caohd96}
Huai-Dong Cao.
\newblock Existence of gradient {K}\"ahler-{R}icci solitons.
\newblock In {\em Elliptic and parabolic methods in geometry (Minneapolis, MN,
  1994)}, pages 1--16. A K Peters, Wellesley, MA, 1996.

\bibitem{caohd97limits}
Huai-Dong Cao.
\newblock Limits of solutions to the {K}\"{a}hler-{R}icci flow.
\newblock {\em J. Differential Geom.}, 45(2):257--272, 1997.

\bibitem{caohd09}
Huai-Dong Cao.
\newblock Recent progress on {R}icci solitons.
\newblock {\em Adv. Lect. Math.}, 11:1--38, 2009.

\bibitem{caozhou10}
Huai-Dong Cao and Detang Zhou.
\newblock On complete gradient shrinking {R}icci solitons.
\newblock {\em J. Differential Geom.}, 85(2):175--185, 2010.

\bibitem{caotran1}
Xiaodong Cao and Hung Tran.
\newblock The {W}eyl tensor of gradient {R}icci solitons.
\newblock {\em Geom. Topol.}, 20(1):389--436, 2016.

\bibitem{CV96}
Thierry Chave and Galliano Valent.
\newblock On a class of compact and non-compact quasi-{E}instein metrics and
  their renormalizability properties.
\newblock {\em Nuclear Phys. B}, 478(3):758--778, 1996.

\bibitem{CZ12Kahler}
Qiang Chen and Meng Zhu.
\newblock On rigidity of gradient {K}\"{a}hler-{R}icci solitons with harmonic
  {B}ochner tensor.
\newblock {\em Proc. Amer. Math. Soc.}, 140(11):4017--4025, 2012.

\bibitem{CF16conical}
Otis Chodosh and Frederick Tsz-Ho Fong.
\newblock Rotational symmetry of conical {K}\"{a}hler-{R}icci solitons.
\newblock {\em Math. Ann.}, 364(3-4):777--792, 2016.

\bibitem{chow}
Bennett Chow.
\newblock The {R}icci flow on the $2$-sphere.
\newblock {\em J. Differential Geom.}, 33(2):325--334, 1991.

\bibitem{chowluni}
Bennett Chow, Peng Lu, and Lei Ni.
\newblock {\em Hamilton's {R}icci flow}, volume~77 of {\em Graduate Studies in
  Mathematics}.
\newblock American Mathematical Society, Providence, RI, 2006.

\bibitem{CCD22finite}
Charles Cifarelli, Ronan~J Conlon, and Alix Deruelle.
\newblock On finite time type {I} singularities of the {K}\"{a}hler-{R}icci
  flow on compact {K}\"{a}hler surfaces.
\newblock {\em J. Eur. Math. Soc.}, to appear.

\bibitem{CD2020expanding}
Ronan~J Conlon and Alix Deruelle.
\newblock Expanding {K}{\"a}hler--{R}icci solitons coming out of {K}{\"a}hler
  cones.
\newblock {\em Journal of Differential Geometry}, 115(2):303--365, 2020.

\bibitem{CDSexpandshriking19}
Ronan~J Conlon, Alix Deruelle, and Song Sun.
\newblock Classification results for expanding and shrinking gradient {K}\"
  {a}hler-{R}icci solitons.
\newblock {\em Geom. Topol.}, 28:267--351, 2024.

\bibitem{DW11coho}
Andrew~S. Dancer and McKenzie~Y. Wang.
\newblock On {R}icci solitons of cohomogeneity one.
\newblock {\em Ann. Global Anal. Geom.}, 39(3):259--292, 2011.

\bibitem{DZ2020rigidity}
Yuxing Deng and Xiaohua Zhu.
\newblock Rigidity of {$\kappa$}-noncollapsed steady {K}\"{a}hler-{R}icci
  solitons.
\newblock {\em Math. Ann.}, 377(1-2):847--861, 2020.

\bibitem{EW00ivp}
J.-H. Eschenburg and McKenzie~Y. Wang.
\newblock The initial value problem for cohomogeneity one {E}instein metrics.
\newblock {\em J. Geom. Anal.}, 10(1):109--137, 2000.

\bibitem{fik03}
Mikhail Feldman, Tom Ilmanen, and Dan Knopf.
\newblock Rotationally symmetric shrinking and expanding gradient
  {K}\"{a}hler-{R}icci solitons.
\newblock {\em J. Differential Geom.}, 65(2):169--209, 2003.

\bibitem{H3}
Richard~S. Hamilton.
\newblock Three-manifolds with positive {R}icci curvature.
\newblock {\em J. Differential Geom.}, 17(2):255--306, 1982.

\bibitem{Hsurface}
Richard~S. Hamilton.
\newblock The {R}icci flow on surfaces.
\newblock In {\em Mathematics and general relativity (Santa Cruz, CA, 1986)},
  pages 237--262. Amer. Math. Soc., Providence, RI, 1988.

\bibitem{hano55}
Jun-ichi Hano.
\newblock On affine transformations of a {R}iemannian manifold.
\newblock {\em Nagoya Math. J.}, 9:99--109, 1955.

\bibitem{KW74curv}
Jerry~L. Kazdan and F.~W. Warner.
\newblock Curvature functions for open {$2$}-manifolds.
\newblock {\em Ann. of Math. (2)}, 99:203--219, 1974.

\bibitem{kobayashi95}
Shoshichi Kobayashi.
\newblock {\em Transformation groups in differential geometry}.
\newblock Classics in Mathematics. Springer-Verlag, Berlin, 1995.
\newblock Reprint of the 1972 edition.

\bibitem{KNauto57}
Shoshichi Kobayashi and Katsumi Nomizu.
\newblock On automorphisms of a {K}\"{a}hlerian structure.
\newblock {\em Nagoya Math. J.}, 11:115--124, 1957.

\bibitem{KNvolumeI96}
Shoshichi Kobayashi and Katsumi Nomizu.
\newblock {\em Foundations of differential geometry. {V}ol. {I}}.
\newblock Wiley Classics Library. John Wiley \& Sons, Inc., New York, 1996.
\newblock Reprint of the 1963 original, A Wiley-Interscience Publication.

\bibitem{KNvolumeII96}
Shoshichi Kobayashi and Katsumi Nomizu.
\newblock {\em Foundations of differential geometry. {V}ol. {II}}.
\newblock Wiley Classics Library. John Wiley \& Sons, Inc., New York, 1996.
\newblock Reprint of the 1969 original, A Wiley-Interscience Publication.

\bibitem{koi90}
Norihito Koiso.
\newblock On rotationally symmetric {H}amilton's equation for
  {K}\"ahler-{E}instein metrics.
\newblock In {\em Recent topics in differential and analytic geometry},
  volume~18 of {\em Adv. Stud. Pure Math.}, pages 327--337. Academic Press,
  Boston, MA, 1990.

\bibitem{Kotschwar18Kahlercone}
Brett Kotschwar.
\newblock K\"{a}hlerity of shrinking gradient {R}icci solitons asymptotic to
  {K}\"{a}hler cones.
\newblock {\em J. Geom. Anal.}, 28(3):2609--2623, 2018.

\bibitem{LNW16fourPIC}
Xiaolong Li, Lei Ni, and Kui Wang.
\newblock Four-dimensional gradient shrinking solitons with positive isotropic
  curvature.
\newblock {\em Int. Math. Res. Not. IMRN}, (3):949--959, 2018.

\bibitem{LW23}
Yu~Li and Bing Wang.
\newblock On {K}\"{a}hler {R}icci shrinker surfaces.
\newblock {\em arXiv preprint arXiv:2301.09784}, 2023.

\bibitem{lich54}
Andr\'{e} Lichnerowicz.
\newblock Sur les groupes d'automorphismes de certaines vari\'{e}t\'{e}s
  {K}\"{a}hleriennes.
\newblock {\em C. R. Acad. Sci. Paris}, 239:1344--1346, 1954.

\bibitem{munse09}
Ovidiu Munteanu and Natasa Sesum.
\newblock On gradient {R}icci solitons.
\newblock {\em J. Geom. Anal.}, 23(2):539--561, 2013.

\bibitem{MW15topo}
Ovidiu Munteanu and Jiaping Wang.
\newblock Topology of {K}\"{a}hler {R}icci solitons.
\newblock {\em J. Differential Geom.}, 100(1):109--128, 2015.

\bibitem{MW17compact}
Ovidiu Munteanu and Jiaping Wang.
\newblock Positively curved shrinking {R}icci solitons are compact.
\newblock {\em J. Differential Geom.}, 106(3):499--505, 2017.

\bibitem{naber07}
Aaron Naber.
\newblock Noncompact shrinking four solitons with nonnegative curvature.
\newblock {\em Journal f{\"u}r die reine und angewandte Mathematik (Crelles
  Journal)}, 2010(645):125--153, 2010.

\bibitem{NN57}
A.~Newlander and L.~Nirenberg.
\newblock Complex analytic coordinates in almost complex manifolds.
\newblock {\em Ann. of Math. (2)}, 65:391--404, 1957.

\bibitem{ONeill66}
Barrett O'Neill.
\newblock The fundamental equations of a submersion.
\newblock {\em Michigan Math. J.}, 13:459--469, 1966.

\bibitem{PTV99quasi}
Henrik Pedersen, Christina T\o~nnesen Friedman, and Galliano Valent.
\newblock Quasi-{E}instein {K}\"{a}hler metrics.
\newblock {\em Lett. Math. Phys.}, 50(3):229--241, 1999.

\bibitem{perelman1}
Grisha Perelman.
\newblock The entropy formula for the {R}icci flow and its geometric
  applications.
\newblock {\em arXiv:math.DG/0211159}, 2002.

\bibitem{perelman3}
Grisha Perelman.
\newblock Finite extinction time for the solutions to the {R}icci flow on
  certain three-manifolds.
\newblock {\em arXiv:math.DG/0307245}, 2003.

\bibitem{perelman2}
Grisha Perelman.
\newblock Ricci flow with surgery on three-manifolds.
\newblock {\em arXiv:math.DG/0303109}, 2003.

\bibitem{pe06book}
Peter Petersen.
\newblock {\em Riemannian geometry}, volume 171.
\newblock Springer, 2006.

\bibitem{PW09grsym}
Peter Petersen and William Wylie.
\newblock On gradient {R}icci solitons with symmetry.
\newblock {\em Proc. Amer. Math. Soc.}, 137(6):2085--2092, 2009.

\bibitem{pewy10}
Peter Petersen and William Wylie.
\newblock On the classification of gradient {R}icci solitons.
\newblock {\em Geom. Topol.}, 14(4):2277--2300, 2010.

\bibitem{schoen95conformal}
R.~Schoen.
\newblock On the conformal and {CR} automorphism groups.
\newblock {\em Geom. Funct. Anal.}, 5(2):464--481, 1995.

\bibitem{tanno69}
Sh\^{u}kichi Tanno.
\newblock The automorphism groups of almost contact {R}iemannian manifolds.
\newblock {\em Tohoku Math. J. (2)}, 21:21--38, 1969.

\bibitem{tanno69Hermitian}
Sh\^{u}kichi Tanno.
\newblock The automorphism groups of almost {H}ermitian manifolds.
\newblock {\em Trans. Amer. Math. Soc.}, 137:269--275, 1969.

\bibitem{tanno69sasa}
Sh\^{u}kichi Tanno.
\newblock Sasakian manifolds with constant {$\phi$}-holomorphic sectional
  curvature.
\newblock {\em Tohoku Math. J. (2)}, 21:501--507, 1969.

\bibitem{tian1}
Gang Tian and Xiaohua Zhu.
\newblock Uniqueness of {K}\"ahler-{R}icci solitons.
\newblock {\em Acta Math.}, 184(2):271--305, 2000.

\bibitem{tubook11}
Loring~W. Tu.
\newblock {\em An introduction to manifolds}.
\newblock Universitext. Springer, New York, second edition, 2011.

\bibitem{WZ04toric}
Xu-Jia Wang and Xiaohua Zhu.
\newblock K\"{a}hler-{R}icci solitons on toric manifolds with positive first
  {C}hern class.
\newblock {\em Adv. Math.}, 188(1):87--103, 2004.

\bibitem{wyliepersonal}
William Wylie.
\newblock Personal communication.

\bibitem{zhang09completeness}
Zhu-Hong Zhang.
\newblock On the completeness of gradient {R}icci solitons.
\newblock {\em Proc. Amer. Math. Soc.}, 137(8):2755--2759, 2009.

\end{thebibliography}

\end{document}